\documentclass[10pt]{amsart}

\usepackage{amsthm}
\usepackage[english]{babel} 
\usepackage{fullpage}
\usepackage{afterpage}
\usepackage{hyperref}
\usepackage[T1]{fontenc}
\usepackage[utf8]{inputenc}
\usepackage{lmodern}
\usepackage{amsmath,amsthm,amsfonts}
\usepackage{mathrsfs} 
\usepackage{graphicx} 
\usepackage{multicol} 
\usepackage{tikz} 
\usepackage{tikz-cd} 
\usepackage[font=small]{caption}
\usepackage{pgfplots} 
\usepackage{wrapfig}
\usepackage{amssymb}
\usepackage{array}
\usepackage{stackrel}
\usepackage{extarrows}
\usepackage{pgfplots}
\usetikzlibrary{decorations.markings}
\usepackage{mathtools}

\usepackage{amsmath,amsthm,amsfonts}
\usepackage{mathrsfs} 
\usepackage{graphicx} 
\usepackage{multicol} 
\usepackage{tikz} 
\usepackage{tikz-cd} 
\usepackage{pgfplots} 
\usepackage{wrapfig}
\usepackage{amssymb}
\usepackage{array}
\usepackage{stackrel}
\usepackage{extarrows}
\usepackage{pgfplots}
\usepgfplotslibrary{polar}
\pgfplotsset{compat=newest}
\usetikzlibrary{calc}
\usetikzlibrary{graphs}
\usetikzlibrary{fit}
\usepgfplotslibrary{polar}
\pgfplotsset{compat=newest}
\usetikzlibrary{calc}

\newcommand{\Z}{\mathbb{Z}}
\newcommand{\G}{\Gamma}

\newcommand{\R}{\mathcal{R}}

\newcommand{\A}{{A_\Gamma}}
\newcommand{\ua}{\underline{A}}
\newcommand{\ub}{\underline{B}}
\newcommand{\us}{\underline{S}}
\newcommand{\up}{\underline{P}}

\newcommand{\ut}{\underline{T}}
\newcommand{\uua}{\underline{\underline{A}}}
\newcommand{\uub}{\underline{\underline{B}}}

\newcommand{\Wh}{\mathrm{Wh}}

\newcommand{\red}{\mathrm{red}}
\newcommand{\MM}{\mathrm{MM}}
\newcommand{\supp}{\mathrm{supp}}
\newcommand{\cd}{\mathrm{cd}}
\newcommand{\cro}{\mathrm{cr}}
\newcommand{\lk}{\mathrm{lk}}
\newcommand{\st}{\mathrm{st}}
\newcommand{\Stab}{\mathrm{Stab}}
\newcommand{\Inn}{\mathrm{Inn}}
\newcommand{\PAut}{\Sigma\mathrm{PAut}}
\newcommand{\POut}{\Sigma\mathrm{POut}}
\newcommand{\Aut}{\mathrm{Aut}}
\newcommand{\Out}{\mathrm{Out}}

\newcommand{\z}{\underline{Z}}

\newcommand{\OO}{ \mathcal{O}}
\newcommand{\boundellipse}[3]
{(#1) ellipse (#2 and #3)}
\newtheorem{teo}{Theorem}[section]
\newtheorem{prop}[teo]{Proposition}

\newtheorem{lem}[teo]{Lemma}

\newtheorem*{teoA}{Theorem A}
\newtheorem*{teoB}{Theorem B}
\newtheorem*{teoC}{Theorem C}
\newtheorem*{teoD}{Theorem D}

\theoremstyle{definition}
\newtheorem{defi}[teo]{Definition}
\newtheorem{ej}[teo]{Example}

\theoremstyle{remark}
\newtheorem{nota}[teo]{Remark}

\title{The McCullough-Miller complex for right angled Artin groups}
\author{Peio Ardaiz Galé}
\author{Conchita Martínez Pérez}
\date{today}



\thanks{\noindent
Both authors are partially supported by the Spanish Government PID2021-126254NB-I00. The first named author is also supported by Universidad P{\'u}blica de Navarra (grant code: Plan de promoci{\'o}n de grupos. ``{\'A}lgebra. Aplicaciones''). The second named author is also supported by Departamento de Ciencia, Universidad y Sociedad del 
Conocimiento del Gobierno de Arag{\'o}n (grant code: E22-23R: ``{\'A}lgebra y Geometr{\'i}a'')}

\begin{document}

\begin{abstract}

McCullough and Miller  constructed a contractible complex on which the pure symmetric automorphism group of a free group acts with free abelian stabilizers. 
This complex has been used for computations such as the cohomological dimension of these groups, their cohomology rings, and results about $\ell^2$-Betti numbers or BNRS-invariants. We generalize this construction to pure symmetric automorphism groups of arbitrary RAAGs and exhibit applications of this generalization.
\end{abstract}

\maketitle

\section{Introduction}
Right-angled Artin groups (RAAGs) are a fascinating family of groups that attracts considerable attention in current research in geometric group theory. Given  a finite simplicial graph \( \Gamma \) (i.e., a graph with no loops or double edges), the associated right angled Artin group is the group defined by the presentation 
\[
A_\Gamma = \langle V_\Gamma : [v, w] = 1, (v, w) \in E(\Gamma) \rangle.
\]
These groups interpolate between free abelian and free groups and their (outer) automorphism groups interpolate between $\mathrm{GL}_n(\Z)$ and $\Aut(F_n)$ (or $\Out(F_n)$ for the outer version). Automorphisms of a RAAG $\A$ that map each standard generator to a conjugate of itself are called {\sl pure symmetric} and they form a subgroup denoted $\PAut(\A)$ that contains the subgroup of inner automorphisms, so there is also an outer version $\POut(\A)$. In the case when $A_\G$ is the free group $F_n$, $\PAut(F_n)$ is the fundamental group of the configuration space of $n$ unknotted, unlinked circles in the sphere $S^3$ \cite{{JMcCM}}. There is also a version of this for pure symmetric automorphisms of other right-angled Artin groups \cite{BB22}. Lawrence showed that the groups $\PAut(A_\G)$ are finitely generated by partial conjugations and later Toinet  \cite{Toi} gave a finite presentation that was subsequently refined by Koban-Piggott \cite{KoPi}. The groups $\POut(\A)$  lie inside the {\sl Torelli} subgroup, which is the kernel of the epimorphism $\Out(A_\G)\to\mathrm{GL}(A_\G/A'_\G)$ induced by the action on the abelianization and therefore they are residually torsion free nilpotent \cite{Toi}. In particular, this implies that $\POut(A_\G)$ is torsion free. For any $A_\G$, it is possible to construct a new graph $\Delta$ such that,  $\A\cong\POut(A_\Delta)$ \cite{Wied}, on the other hand, not all $\POut(A_\G)$ are RAAGs \cite{DayWade}.

In \cite{MM}, McCullough and Miller construct a simplicial contractible complex that admits an action of the group of  symmetric automorphisms of a free product. In the particular case of a free group, this is a complex $\MM_n$ admitting an action of the group \( \POut(F_n) \) with free abelian stabilizers. The construction is combinatorial, based on the Whitehead poset, and the argument given there to show contractibility is particularly intricate. This complex does not seem to admit a CAT(0) geometry \cite{P}. The complex $\MM_n$ has subsequently been used to prove that $\POut(F_n)$ is a duality group \cite{NB},  to compute the cohomology ring \( \mathrm{H}^*(\POut(F_n),\mathbb{Z} \)) \cite{JMcCM}, or to characterize certain properties of the BNRS-invariants of the groups $\POut(F_n)$ \cite{ErshovZaremsky}.

A very fruitful line of research in geometric group theory is to extend various results and constructions from free groups to RAAGs, a remarkable example is the generalization by Bregman, Charney and Vogtmann \cite{BChK}  of the Culler-Vogtmann  Outer space for $\Out(F_n)$ \cite{CV}.
Following this line, in this paper, we generalize the construction of the McCullough-Miller complex to pure symmetric automorphism groups of arbitrary RAAGs.   This involves finding a $\G$-version $\Wh_\G$ of the Whitehead poset (see Section \ref{sec:construction}). Once that is done, we define the $\Gamma$-McCullough Miller complex $\MM_\G$ as the simplicial realization of the poset of equivalence classes of marked elements of the $\G$-Whitehead poset. The first properties of this construction follow easily and we get:

\begin{teoA} The group $\POut(A_\G)$ acts on $\MM_\G$
    with free abelian stabilizers and fundamental domain the simplicial realization of the $\G$-Whitehead poset $\Wh_\G$.
\end{teoA}

As with the original construction by McCullough and Miller, the most complicated part is proving contractibility. To do that, we use two main tools: on the one hand, we extend to our setting many of the extremely useful technical definitions given by McCullough and Miller and on the other we use results of Day \cite{BDay} to get an analysis of the behavior of lengths of cyclic word representatives in RAAGs under the action of pure symmetric automorphisms. This analysis needs more sophisticated tools than in the case of the original construction. Moreover, our argument fixes a particular issue with one crucial technical Lemma in \cite{MM} (see Example \ref{ex:issue}). Once all this is done, the proof of contractibility can be performed along similar lines as in \cite{MM} and we get our main theorem.

\begin{teoB} The complex $\MM_\G$ is contractible.
\end{teoB}

Actions of groups on contractible complexes with nice cell stabilizers have many applications. As a well known example we may think of the action of a right angled Artin group on the Deligne complex, where cell stabilizers are parabolic subgroups, i.e., subgroups conjugated to the subgroup generated by complete subgraphs of the defining graph and as such are obviously free abelian.  In our case, cell stabilizers are also free abelian so the complex $\MM_\G$ can be seen as a Deligne-type complex for the groups $\POut(A_\G)$.  

We exhibit two applications of our complex. First, we show how to use the action of $\POut(A_\G)$ on $\MM_\G$ to compute the cohomological dimension of the group in terms of a rank notion that we define for elements of the Whitehead poset (see Section \ref{sec:applications}). This result generalizes the fact that $\POut(F_n)$ has cohomological dimension $n-2$ \cite{Collins}. In \cite{DayWade2} there is an algorithm to compute the virtual cohomological dimension of the whole group $\Aut(A_\G)$ of automorphisms of a RAAG. As the authors point out, that algorithm can also be used for certain relative versions of $\Out(A_\G)$ that include the subgroups $\POut(A_\G)$. However, that algorithm is complicated to apply. In the same paper, the authors show that in some cases the virtual cohomological dimension is realized by the rank of a free abelian subgroup, in the sense that there is a free abelian subgroup with rank the virtual cohomological dimension of the ambient group. We show that this is also true for the groups $\POut(A_\G)$.

\begin{teoC} The cohomological dimension of $\POut(A_\G)$ is 
$$\cd(\POut(A_\G))=\max\{r(\uua)\mid\uua\in \Wh_\G\}$$ 
and can be realized by a free abelian subgroup, that is, it is equal to the maximal rank of a free abelian subgroup of $\POut(A_\G)$. 
\end{teoC}

 Recently, Corrigan has also obtained Theorem C by different methods. He constructs a contractible cube complex with a proper action of the group of symmetric outer automorphisms $\Sigma\Out(A_\Gamma)$ which is a finite index extension of our group $\POut(A_\Gamma)$ \cite{Corrigan}.   Abgrall has extended that construction to other automorphisms groups. Moreover, these complexes realize the cohomological dimension of the corresponding groups. In fact,  one can see Corrigan's and Abgrall's complexes as a version of the Salvetti complex for these groups, whereas, as mentioned above, ours would correspond to the Deligne complex.

For our next application, we consider the $\ell^2$-cohomology. McCammond and Meier used the action of $\POut(F_n)$ on the McCullough-Miller complex \cite[Theorem 8.2]{McCM} to describe the $\ell^2$ cohomology groups of $\POut(F_n)$. With the complex MM$_\G$ one gets the same result under an extra assumption on centralizers; and a weaker version for the general case:

\begin{teoD} For a right-angled Artin group $A_\G$, the von Neumann dimension of
$H^i(\MM _\G;\mathcal{N}(\POut(A_\G))$
equals the von Neumann dimension of 
$$\mathcal{N}(\POut(A_\G))\otimes\overline{H}^{i-1}(|\Wh_\G^0|)$$ where $\mathcal{N}(\POut(A_\G))$ is the group von Neumann algebra.

In addition, if $\POut(A_\G)$ has no non-trivial element whose centralizer has finite index in $\POut(A_\G)$, then the $i$-th $\ell^2$-cohomology group of $\POut(A_\G)$ is isomorphic to
$$\ell^2(\POut(A_\G))\otimes\overline{\mathrm{H}}^{i-1}(|\Wh_\G^0|)$$
\end{teoD}

Here, $\Wh_\G^0\leq \Wh_\G$ is the subposet of the $\G$-Whitehead poset obtained after removing the nuclear vertex, so $|\Wh_\G|$ is the cone over $|\Wh_\G^0|$.

\medskip

The paper is organized as follows: in Section \ref{sec:MMfree} we review the construction of the McCullough-Miller complex for $\POut(F_n)$. This section is not strictly necessary but it can be helpful for the rest of the paper.
In Section \ref{sec:puresym} we recall some properties of the groups $\PAut(A_\G)$ that will be needed in the sequel. For example, we describe the generating set given by partial conjugations given in \cite{Lau} and the associated presentation \cite{Toi}, \cite{KoPi}.
The construction of the $\Gamma$-Whitehead poset and of the complex $\MM_\G$ is carried out in Section \ref{sec:construction}.
Sections \ref{sec:notation} and \ref{sec:reductivity} are technical. In Section \ref{sec:notation}, we extend important notions from \cite{MM} such as compatibility, refinement, disjunction; in Section \ref{sec:reductivity} we extend the notion of reductivity. Some of the results of \cite{MM} can be translated without major changes to our setting, which is mostly the case in Section \ref{sec:notation}. However, results on reductivity (Section \ref{sec:reductivity}) require a different approach, partly because of the issue mentioned above. To achieve this, we make a crucial use of results of \cite{BDay} in this Section. We also prove Theorem A in Section \ref{sec:notation}. Theorem B is shown in Section \ref{sec:contractibility}. The last part of our argument follows the lines of the proof of contractibility of the complex $\MM_n$ in \cite{MM}, but we provide alternative, shorter arguments at some points. Finally, in Section \ref{sec:applications} we prove Theorem C and Theorem D.

{\sl Acknowledgments:} We would like to thank Richard Wade for useful conversations, and Matt Day for his help with how to use the results in \cite{BDay} in our setting, and in particular for pointing out the argument in the proof of Lemma \ref{lem:existence}. We also thank Gabriel Corrigan for conversations regarding \cite{Corrigan} and the anonymous referee for many suggestions that contributed to improve the paper.

\section{The McCullough-Miller Complex}\label{sec:MMfree}
Let $F_n$ be the free group of rank $n$ with basis $X$, and let $\PAut(F_n)$ denote the group of pure symmetric automorphisms of $F_n$, i.e., the subgroup of Aut$(F_n)$ consisting of automorphisms that map each element of $X$ to a conjugate of itself. This group contains the subgroup of inner automorphisms and the corresponding outer version is denoted $\POut(F_n)=\PAut(F_n)/\Inn(F_n)$. In \cite{MM} McCullough and Miller introduced a family of contractible complexes denoted MM$_n$
which admit an action by $\POut(F_n)$. This action is not free but stabilizers are free abelian. Moreover, there is a fundamental domain that can be described combinatorially as follows. 

\subsection{Labeled bipartite trees, vertex types and based partitions}
We begin by defining the elements that will be the vertices of the space that we are constructing.

\begin{defi}[$n$-labeled bipartite trees or vertex types]
An {\sl $[n]$-vertex type} or {\sl $[n]$-labeled bipartite tree} (we will often omit $[n]$) is a tree $T$  together with an embedding $\iota$ from $[n]=\{1,\ldots,n\}$ to the set of nodes of $T$ satisfying the following two conditions. Elements in the image of $\iota$ are called labeled nodes, and all other nodes are called unlabeled. Moreover:

        \begin{itemize}
            \item [1)] each edge of $T$ has exactly one labeled endpoint, and
            \item[2)] each node with valence 1 is labeled.
        \end{itemize}
    
\end{defi}
We say that two vertex types are equivalent if there is a label preserving graph isomorphism between them. The {\sl rank} of a vertex type $T$ is $m-1$ where $m$ is the number of unlabeled nodes in $T$. There is only one $[n]$-vertex type of rank 0, namely the tree with a single unlabeled node of valence $n$. This is called the {\sl nuclear vertex} and is denoted by $\OO$. In Figure \ref{fig:ArbEtEJ} we have represented possible vertex types for $n=4$; by permuting the labels of the vertex types $A, B$ and $C$, we obtain all the possible vertex types.

\begin{figure}[h]
    \centering
    \begin{tikzpicture}
     \node[shape=circle,draw=black, fill=black, label={above left: 1}, scale=0.75] (1) at (-4,3) {};
    \node[shape=circle,draw=black, fill=black,  label={above right: 2}, scale=0.75] (2) at (-2,3) {};
    \node[shape=circle,draw=black, fill=black, label={below right: 3}, scale=0.75] (3) at (-2,1) {};
    \node[shape=circle,draw=black, fill=black, label={below left: 4}, scale=0.75] (4) at (-4,1) {};
    \node[shape=circle,draw=black, fill=black, scale=0.5] (A) at (-3,2) {};

    \path (1) edge (A);
    \path (3) edge (A);
    \path (2) edge (A);
    \path (4) edge (A);

    \node at (-4.7,2) {$\OO =$};

     \node[shape=circle,draw=black, fill=black, label={above : 1}, scale=0.75] (5) at (0.5,2) {};
    \node[shape=circle,draw=black, fill=black, label={above : 2}, scale=0.75] (6) at (2.5,2) {};
    \node[shape=circle,draw=black, fill=black, label={above: 3}, scale=0.75] (7) at (4.5,2) {};
    \node[shape=circle,draw=black, fill=black, label={above: 4}, scale=0.75] (8) at (6.5,2) {};
    \node[shape=circle,draw=black, fill=black, scale=0.5] (B) at (1.5,2) {};
    \node[shape=circle,draw=black, fill=black, scale=0.5] (C) at (3.5,2) {};
    \node[shape=circle,draw=black, fill=black, scale=0.5] (D) at (5.5,2) {};

     \path (5) edge (B);
    \path (6) edge (B);
    \path (6) edge (C);
    \path (7) edge (C);
    \path (8) edge (C);

    \node at (-0.2,2) {$B =$};

    \node[shape=circle,draw=black, fill=black, label={above left: 1}, scale=0.75] (9) at (-4,-0.5) {};
    \node[shape=circle,draw=black, fill=black,  label={above : 4}, scale=0.75] (10) at (0,-1.5) {};
    \node[shape=circle,draw=black, fill=black, label={above: 3}, scale=0.75] (11) at (-2,-1.5) {};
    \node[shape=circle,draw=black, fill=black, label={below left: 2}, scale=0.75] (12) at (-4,-2.5) {};
    \node[shape=circle,draw=black, fill=black, scale=0.5] (E) at (-3,-1.5) {};
    \node[shape=circle,draw=black, fill=black, scale=0.5] (F) at (-1,-1.5) {};

    \path (9) edge (E);
    \path (12) edge (E);
    \path (11) edge (F);
    \path (11) edge (E);
    \path (10) edge (F);

    \node at (-4.7,-1.5) {$A =$};

     \node[shape=circle,draw=black, fill=black, label={above : 1}, scale=0.75] (13) at (2.5,-1.5) {};
    \node[shape=circle,draw=black, fill=black, label={below : 3}, scale=0.75] (14) at (4.5,-1.5) {};
    \node[shape=circle,draw=black, fill=black, label={above: 4}, scale=0.75] (15) at (6.5,-1.5) {};
    \node[shape=circle,draw=black, fill=black, label={above: 2}, scale=0.75] (16) at (4.5,0.5) {};
    \node[shape=circle,draw=black, fill=black, scale=0.5] (G) at (3.5,-1.5) {};
    \node[shape=circle,draw=black, fill=black, scale=0.5] (H) at (5.5,-1.5) {};
    \node[shape=circle,draw=black, fill=black, scale=0.5] (I) at (4.5,-0.5) {};

    \path (13) edge (G);
    \path (14) edge (G);
    \path (14) edge (I);
    \path (14) edge (H);
    \path (16) edge (I);
    \path (15) edge (H);

    \node at (1.8,-1.5) {$C =$};
    
    \end{tikzpicture}
   \caption{Several [4]-vertex types}
    \label{fig:ArbEtEJ}
\end{figure}

For each $[n]$-vertex type $T$ and each $i\in[n]$ we may define a partition $\ut_i$ of the set $[n]$ as follows: remove the vertex labeled $i$ from $T$, and then consider the sets of labels of each of the resulting connected components. We add to this family the set $\{i\}$, and we get the partition $\ut_i$ where the set $\{i\}$ has a distinguished role: we will say that it is the {\sl operative factor} of the partition. These kinds of partitions, i.e., partitions of the set $[n]$ having a distinguished one-element set like $\ut_i$, will be called {\sl based partitions}. We will always place the operative factor in the first position. Following \cite{MM}, we will call the rest of elements in the partition {\sl petals}, where the name refers to the fact that a useful way to represent a based partition is as a flower diagram with the operative factor in the middle and the rest of subsets forming the petals. The {\sl length} of a based partition will be the number of petals and based partitions of length one are called {\sl trivial}.

For example,  for the vertex type $B$ of Figure \ref{fig:ArbEtEJ}, we have the following based partitions
\[\begin{aligned}
    \ub_1&=\{\{1\},\{2,3,4\}\},\\
    \ub_2&=\{\{2\},\{1\},\{3,4\}\},\\
\ub_3&=\{\{3\},\{1,2\},\{4\}\},\\
\ub_4&=\{\{4\},\{1,2,3\}\}.\\
\end{aligned}
\]
In Figure \ref{fig:flower}, we can find the flower diagram representing the based partition $\ub_2$.

\begin{figure}
\begin{tikzpicture}
   \centering
        \draw (10,0) node[minimum size=1cm,circle,draw] (O) {2};
       \draw (O.80) to[out=50,in=-10,looseness=10]  node[pos=0.1](C1){} node[pos=0.4](C2){} node[pos=0.8](C3){} node[pos=0.9](C4){}  (O.-40);
\node[fit=(C1)(C2)(C3)(C4)]{$3$\quad$4$};      
 \draw (O.120) to[out=135,in=210,looseness=8]  node[pos=0.1](D1){} node[pos=0.4](D2){} node[pos=0.8](D3){} node[pos=0.9](D4){}  (O.-110);
\node[fit=(D1)(D2)(D3)(D4)]{$1$};      
    \end{tikzpicture}
    \caption{A flower diagram}
    \label{fig:flower}
    \end{figure}
\smallskip

One can recover the vertex type $T$ from the set of based partitions $\ut_1,\ldots,\ut_n$ \cite{MM}. In particular, note that the nuclear vertex corresponds to the family of trivial partitions.

\subsection{The Whitehead poset} Next, we define a partial order on the set of vertex types. To do that, let $T$ be a vertex type. We say that $\widehat T$ is obtained from $T$ by {\sl folding} if $\widehat T$ is the vertex type obtained by identifying a pair of edges $e_1\neq e_2$ of $T$ that share 
a common labeled endpoint.  Folding reduces the rank of a tree by 1. For example, in
Figure \ref{fig:ArbEtEJ}, folding $B$ at 2 yields $A$, folding $A$ at 3 yields $\OO$,
and one of the three possible foldings of $C$ at 3 yields $A$. If $\widehat T$ is obtained from $T$ by a series of foldings, we put $\widehat T<T$.

\begin{defi}[Whitehead poset]
    The Whitehead poset $\Wh_n$ is the poset consisting
of all $[n]$-vertex types under the partial order defined above. Observe that for any vertex type $T$ we have $\OO\leq T$, and the poset distance between $\OO$ and $T$ is precisely the rank of $T$. It is also easy to verify that every maximal chain in the poset $\Wh_n$ has length $n-1$.
\end{defi}

To construct a complex admitting an action of the group $\PAut(F_n)$, we need the notion of markings that we consider next.

\begin{defi}[Markings]\label{def:marking}
    A {\sl marking} of a vertex type $T$ or of a based partition $\ut_i$ is an ordered 
basis $Y=\{y_1,\ldots,y_n\}$ of $F_n$ such that each $y_i$ is a conjugate of
$x_i\in X$. The {\sl marked vertex type} is the pair $(Y,T)$, and the {\sl marked based partition} $(Y,\ut_i)$.  We see the marking as a relabeling of the nodes of $T$ or of $\ut_i$ where $y_j$ is associated with the node $j$, in particular the element $y_i$ is the operative factor of $(Y,\ut_i)$.
\end{defi}

With this definition, observe that the set of marked vertex types is also a poset with $(Y',T')\leq(Y,T)$ if and only if $Y'=Y$ and $T'\leq T$. Also, there is an obvious action of $\PAut(F_n)$ on this poset via $\alpha(Y,T)=(\alpha(Y),T)$ for $\alpha\in\PAut(F_n)$. However, in order to get a space with an action of the outer automorphism group, we need to construct a quotient of this poset.

\begin{defi}[Carried automorphisms]\label{def:carriedFree} Let $(Y,\ut_i)$ be a marked based partition with marking $Y=\{y_1,\ldots,y_n\}$ and operative factor $y_i$. We say that $\alpha\in\Aut(F_n)$ is carried by $(Y,\ut_i)$ if
     \begin{itemize}
         \item[i)] $\alpha(y_i)=y_i$, i.e., the operative factor gets fixed.
         \item[ii)] For all $ y_j$, $\alpha(y_j)$ is the result of conjugating $y_j$ by a power of $y_i$.
         \item[iii)] If $y_j,y_k$ lie in the same petal of $\underline{T}_i$, then they get conjugated by the same power of $y_i$ when applying $\alpha$.
     \end{itemize}
     We say that $\underline{T}_i$ is a \textit{full carrier} of $\alpha$ if it carries $\alpha$ and each petal gets conjugated by a different power of $y_i$. We also say that an automorphism $\alpha$ is {\sl carried by the vertex type} $(Y,T)$ if $\alpha=\alpha_1\cdots\alpha_k$, where each $\alpha_j$ is carried by some based partition $(Y,\underline{T_j})$ of $(Y,T)$.
\end{defi}

We can use the notion of carried automorphisms to define an equivalence relationship $\equiv$ in the set of marked vertex types as the relation generated by 
$$(Y_1,T)\equiv(Y_2,T)\text{ if there is $\alpha\in\PAut(F_n)$ carried by $(Y_1,T)$ such that $\alpha(Y_1,T)=(Y_2,T)$}.$$
The equivalence class of an element $(Y,T)$ will be denoted $[Y,T]$.
This equivalence relationship is compatible with the order defined above in the following sense. If $(Y,T')\leq(Y,T)$ and $(Y,T')\equiv(Y_1,T')$, then there is some $\alpha$ carried by $(Y,T)$ with $\alpha(Y,T')=(Y_1,T')$. The set of automorphisms carried by $(Y,T')$ is inside the set of automorphisms carried by $(Y,T)$ so we also have $(Y,T)\equiv(Y_1,T)$. This can be used to show that $\leq$ induces a well defined poset structure in the set of equivalence classes.

\begin{defi}[The McCullough-Miller poset and complex]
    The McCullough-Miller poset  is the poset of marked $[n]$-vertex types, modulo the equivalence relation $\equiv$ defined above.
        The McCullough-Miller complex $\MM_n$ is the simplicial realization of the McCullough-Miller poset.
\end{defi}

\begin{teo}[McCullough-Miller \cite{MM}]
    The complex $\MM_n$ is a contractible complex of dimension $n-2$.
\end{teo}

The statement about the dimension is a direct consequence of the fact, stated above, that maximal chains in the poset have length $n-1$. The complex constructed in the following sections is a generalization of this space.

The action of the group $\PAut(F_n)$ on the set of marked vertex types yields a well defined action on $\MM_n$ that factors through the natural projection $\PAut(F_n)\to\POut(F_n)$. This action has a fundamental
domain which we can identify with the Whitehead poset $|\Wh_n|$, obtained by restricting to classes $[X,T]$ where the marking is the defining
basis. 

The
stabilizer of a vertex $[Y,T]\in\MM_n$ in $\POut(F_n)$
consists of all outer automorphisms that can be expressed as a product of automorphisms
that are carried by $(Y,T)$. Section 5 of \cite{MM} implies the
following:

\begin{lem}[\cite{MM}]
    Under the action of $\POut(F_n)$, the stabilizer of a rank $k$ vertex of $\MM_n$
is a free abelian group of rank $k$.
\end{lem}

\subsection{Crossings and the order relation in \texorpdfstring{$\MM_n$}{} revisited}

In order to generalize the construction of $\MM_n$ to groups of pure symmetric automorphisms of arbitrary RAAGs, we will need to see vertex types in terms of the associated family of based partitions. Therefore, it will be convenient to understand better when a family of based partitions comes from a vertex type, and also how to redefine  the order relation in terms of based partitions.
To do that, here we collect some definitions and facts from \cite{MM}.

\begin{defi}[Crossings and disjoint based partitions]\label{def:cross}
    Let $\ua_i,\ub_j$ be two based partitions with operative factors $i$ and $j$ respectively. Two petals $P\in \ua_i$ and $Q\in \ub_j$ {\sl cross} if $i\neq j$, $P\cap Q\neq \emptyset$, $i\notin Q$ and $j\notin P$. The number of crossings between two based partitions is the number of such pairs $P$ and $Q$ and is denoted $\cro(\ua_i,\ub_j)$. We say that two partitions are disjoint if $\cro(\ua_i,\ub_j)=0$.
\end{defi}

\begin{figure} 
\begin{tikzpicture}
   \centering
        \draw (10,0) node[minimum size=1cm,circle,draw] (O) {3};
       \draw (O.80) to[out=50,in=-10,looseness=10]  node[pos=0.1](C1){} node[pos=0.4](C2){} node[pos=0.8](C3){} node[pos=0.9](C4){}  (O.-40);
\node[fit=(C1)(C2)(C3)(C4)]{$1$\quad \quad$4$};      
 \draw (O.120) to[out=135,in=210,looseness=8]  node[pos=0.1](D1){} node[pos=0.4](D2){} node[pos=0.8](D3){} node[pos=0.9](D4){}  (O.-110);
\node[fit=(D1)(D2)(D3)(D4)]{$8$};      
\draw (10,-1) node {$\ua_3$};
 \draw (13.8,0) node[minimum size=1cm,circle,draw] (P) {5};
       \draw (P.80) to[out=50,in=-10,looseness=10]  node[pos=0.1](E1){} node[pos=0.4](E2){} node[pos=0.8](E3){} node[pos=0.9](E4){}  (P.-40);
\node[fit=(E1)(E2)(E3)(E4)]{$9$};      
 \draw (P.120) to[out=135,in=195,looseness=12]  node[pos=0.0](F1){} node[pos=0.1](F2){} node[pos=0.85](F3){} node[pos=0.9](F4){}  (P.-150);
\node[fit=(F1)(F2)(F3)(F4)]{$7$};   
\draw (13.8,-1) node {$\ub_5$};

\draw (11.9,2) node {2 \quad \quad 6};
    \end{tikzpicture}
     \caption{Two based partitions with a crossing}\label{fig:Cross}
    \end{figure}

To easily see the possible crossings between two based partitions, it is useful to use flower diagrams, but omitting in each flower the petal that contains the operative factor of the other (which we call the \textit{infinity component}). To illustrate this, see Figure \ref{fig:Cross} where the based partitions are $\ua_3=\{\{3\},\{8\},\{1,4\},\{2,5,6,7,9\}\}$ and $\ub_5=\{\{5\},\{4,7\},\{9\},\{1,2,3,6,8\}\}$. 

For the next two lemmas, we refer to \cite{McCM}.

\begin{lem}\label{lem:disjoint}
    Two based partitions $\ua_i$ and $\ub_j$ are disjoint if and only if 
    there exist petals $P\in \ua_i$ and $Q\in \ub_j$ such that $P\cup Q=[n]$. Trivial partitions (those with only one petal) are always disjoint from any other based partition with different operative factor.
\end{lem}

\begin{lem} The based partitions $\ut_i$, $i\in\{1,\ldots,n\}$ associated to a given vertex type $T$ are pairwise disjoint. Conversely, given a family $\{\ut_1,\ldots,\ut_n\}$ of $n$ pairwise disjoint based partitions with different operative factors, they can be realized in a vertex type $T$.
\end{lem}

From now on, we will often denote vertex types with a double underline, for example $\uua$. The reason for doing that is mnemonic, as we have the following layers: vertex types as $\uua$ are sets of based partitions as $\ua$ and based partitions are sets of petals.

\begin{defi}[Order of based partitions] Let $\ua_i$ and $\ub_i$ be based partitions of $[n]$ both with operative factor $i$.
    We say that $\ua_i\leq \ub_i$ if each petal of $\ub_i$ is contained in a petal of $\ua_i$, that is, if each petal of $\ua_i$ is a union of petals of $\ub_i$. In that case note that $l(\ua_i)\leq l(\ub_i)$.
\end{defi}

The fact that this definition is compatible with the order of vertex types can be found in \cite{MM}: 

\begin{lem}
    Let $\uua,\underline{\underline{\widehat A}}\in \Wh_n$. Then $\underline{\underline{\widehat A}}\leq \uua$ if and only if $\underline{\widehat A_i}\leq \ua_i$ for each pair of based partitions $\underline{\widehat A_i}$ of $\underline{\underline{\widehat A}}$, $\ua_i$ of $\uua$.
\end{lem}

\section{Pure symmetric automorphisms of RAAGs}\label{sec:puresym}

    For the rest of the paper we fix a finite simplicial graph $\Gamma$ with set of vertices $V_\G:=V(\Gamma)$ and set of edges $E(\G)$. We view  $V_\G$ as the standard generating set for the right angled Artin group $A_\G$ defined in the introduction.
    Let $v\in V_\G$, $\lk_\G(v)$ denotes the {\sl link} of $v$, which is the subgraph induced by the vertices of $\G$ adjacent to $v$ and $\st_\G(v)$ denotes the {\sl star}, which is the subgraph induced by $\lk_\G(v)\cup\{v\}$. Here, the subgraph of $\Gamma$ {\sl induced} by a set of vertices $\Omega\subseteq V_\G$ is the maximal subgraph of $\G$ with  vertex set $\Omega$.

\begin{defi}[Graphically reduced words]
If $w$ is a word in $L:=V_\G^\pm$, we say that $w$ is {\sl graphically reduced} if there is no subword of the form $(v\nu v^{-1})^\pm$ where $\nu$ is a word in $\st_\G(v)^\pm$. Graphically reduced words are shortest length representatives (\cite[Proposition 1.2]{Gen})  and in particular, if a non empty word is graphically reduced, then it cannot represent the identity.
\end{defi}

\subsection{Pure symmetric automorphisms of RAAGs} As in the free group case, we now define  the notion of pure symmetric automorphisms for RAAGs.

\begin{defi} Let $A_\Gamma$ be a RAAG. The group of pure symmetric automorphisms $\PAut(A_\G)$ of $A_\G$ is the group consisting of those automorphisms of $A_\G$ that map each standard generator of $A_\G$ (i.e., each vertex of $\G$) to a conjugate of itself. Obviously, it contains the inner automorphisms and the corresponding outer version is denoted $$\POut(A_\G):=\PAut(A_\G)/\Inn(A_\G).$$
\end{defi}

Next, we introduce an important family of elements of $\PAut(A_\G)$.

\begin{defi}[Partial conjugations]
Given $a \in V_\G$ and $A$ a union of connected components of $\Gamma-\st_\G(a)$, the partial conjugation of $A$ by $a$ is the automorphism $C_{A}^a$  given by
$$C_A^a(b)=
\left\{\begin{array}{l}
aba^{-1}, \text{ if } b \in A \\
b, \text{ if } b \notin A
\end{array}\right.
$$   
where $b\in\G$. We will also use this notation for partial conjugations where we conjugate by some $a\in V_\G^{-1}$, in fact $C^{a^{-1}}_A=(C^a_A)^{-1}$. We say that the set $A$ is the {\sl support} of $C^a_A$ and we extend this in the obvious way to the case when $\alpha$ is a product of (powers of) partial conjugations by the same $a$ based at different connected components of $\G-\st_\G(a)$. We use $\supp(\alpha)$ for the support of $\alpha$.
\end{defi}

\begin{teo}[Laurence \cite{Lau}] The set of partial conjugations is a generating system for $\PAut(A_\Gamma)$.   
\end{teo}

Observe that if $A_1$ and $A_2$ are unions of connected components of $\Gamma-\st_\G(a)$ and $A_1\cap A_2=\emptyset$, then $C_A^a=C_{A_1}^aC_{A_2}^a$ so in fact to generate $\PAut(A_\Gamma)$ we only need to consider partial conjugations $C_A^a$ where $A$ is a single connected component of $\Gamma-\st_\G(a)$. 

 Toinet \cite{Toi}  computed an explicit presentation of $\PAut(A_\G)$ in terms of partial conjugations. This presentation was later refined by Koban-Piggot \cite{KoPi}; one of the ingredients is the following identity for $C^a_A$ and $C^b_B$ partial conjugations in $\PAut(A_{\Gamma})$ (here $a,b\in V_\G$, $A$ is a union of connected components of $\G-\st_\G(a)$ and $B$ of $\G-\st_\G(b)$):
  $$\left[C_{A}^a, C_{B}^b\right]=1\text{ if either }(A \cup\{a\}) \cap(B \cup\{b\})=\varnothing\text{ or }(A \cup\{a\}) \subseteq B.$$
     It will be helpful below to express these identities and the associated presentation using the following notation
from Day-Wade \cite{DayWade}.

\begin{defi}[Graph components: shared, dominant and subordinate]
     Let $a, b\in\Gamma$ be not adjacent.  A connected component of $\Gamma-\st_\G(a)$ that is also a connected component of $\Gamma-\st_\G(b)$ is said to be {\sl shared}. The unique component of $\Gamma-\st_\G(a)$ that contains $b$ is called the {\sl dominant} component of $\G-\st_\G(a)$ respect to $b$. Finally, a {\sl subordinate} component is any connected component of $\Gamma-\st_\G(a)$ that is contained in the dominant component of $\G-\st_\G(b)$ respect to $a$. By Lemma 2.1 in \cite{DayWade}, any connected component of $\Gamma-\st_\G(a)$ is of one of these types. Observe that if $A$ is a connected component of $\Gamma-\st_\G(a)$ which is either shared or subordinate, then $A\cap\st_\G(b)=\emptyset$.
     \label{defi:comp}
\end{defi}

An example is shown in Figure \ref{fig:Comp}. In this case $\G-\st_\G(b)$ has  dominant component $D_a$, one shared component $C$ and one subordinate component $S$, and $\G-\st_\G(a)$ has  dominant component $D_b$, one shared component $C$ and one subordinate component $S'$.
\begin{figure}
    \centering
    \begin{tikzpicture}
        \node[shape=circle,draw=black, fill=black] (A) at (-10,0) {};
    \node[shape=circle,draw=black, fill=black] (B) at (-9,0) {};
    \node[shape=circle,draw=black, fill=black, label={above: $b$}] (C) at (-8,0) {};
    \node[shape=circle,draw=black, fill=black] (D) at (-7,0) {};
    \node[shape=circle,draw=black, fill=black] (E) at (-6,0) {};
    \node[shape=circle,draw=black, fill=black] (F) at (-5.5,-1) {} ;
    \node[shape=circle,draw=black, fill=black] (G) at (-5,-2) {} ;
    \node[shape=circle,draw=black, fill=black] (H) at (-7.5,-1) {} ;
    \node[shape=circle,draw=black, fill=black] (I) at (-7.5,-2) {} ;

    \draw [red, thick=1cm] plot [smooth cycle, tension=0.9] coordinates {(-10.5,0) (-10,0.5) (-9.5,0) (-10,-0.5)};
    \node at (-10.8,0) {$S$};

    \draw [black, thick=1cm] plot [smooth cycle, tension=0.9] coordinates {(-9.4,0) (-8,0.75) (-6.6,0) (-7.5,-1.35)};
    \node at (-8,1) { $\st_\G(b)$};

     \draw [red, thick=1cm] plot [smooth cycle, tension=0.9] coordinates {(-6.35,0.3) (-5.75,0.3)(-5,-1)  (-4.8,-2.5) (-6,-1) };
    \node at (-4.7,-1) {$D_a$};

    \draw [red, thick=1cm] plot [smooth cycle, tension=0.9] coordinates {(-8,-2) (-7.5,-1.5) (-7,-2) (-7.5,-2.5)};
    \node at (-8.25,-2) { $C$};

    \path (A) edge (B);
    \path (C) edge (B);
    \path (C) edge (D);
    \path (C) edge (H);
    \path (E) edge (H);
    \path (H) edge (I);
    \path (E) edge (D);
    \path (E) edge (F);
    \path (G) edge (F);

     \node[shape=circle,draw=black, fill=black] (A') at (-2,0) {};
    \node[shape=circle,draw=black, fill=black] (B') at (-1,0) {};
    \node[shape=circle,draw=black, fill=black] (C') at (0,0) {};
    \node[shape=circle,draw=black, fill=black] (D') at (1,0) {};
    \node[shape=circle,draw=black, fill=black, label={right: $a$}] (E') at (2,0) {};
    \node[shape=circle,draw=black, fill=black] (F') at (2.5,-1) {} ;
    \node[shape=circle,draw=black, fill=black] (G') at (3,-2) {} ;
    \node[shape=circle,draw=black, fill=black] (H') at (0.5,-1) {} ;
    \node[shape=circle,draw=black, fill=black] (I') at (0.5,-2) {} ;

    \path (A') edge (B');
    \path (C') edge (B');
    \path (C') edge (D');
    \path (C') edge (H');
    \path (E') edge (H');
    \path (H') edge (I');
    \path (E') edge (D');
    \path (E') edge (F');
    \path (G') edge (F');

     \draw [blue, thick=1cm] plot [smooth cycle, tension=0.7] coordinates {(-2.4,0) (-2,0.5) (-1,0.5) (0,0.5) (0.4,0) (0,-0.5) (-1,-0.5) (-2,-0.5)  };
    \node at (-1,0.75) { $D_b$};

    \draw [black, thick=1cm] plot [smooth cycle, tension=0.9] coordinates {(0.68,0.3) (2,0.57) (2.8,0) (2.5,-1.55) (0.35,-1.35)  (0.55, -0.5)};
    \node at (2,0.85) { $\st_\G(a)$};

     \draw [blue, thick=1cm] plot [smooth cycle, tension=0.9] coordinates {(3,-1.5) (3.5,-2)(3,-2.5)  (2.5,-2) };
    \node at (3.75,-2) { $S'$};

    \draw [blue, thick=1cm] plot [smooth cycle, tension=0.9] coordinates {(0.1,-2) (0.5,-1.6) (0.9,-2) (0.5,-2.4)};
    \node at (-0.25,-2) {$C$};

    \end{tikzpicture}
  \caption{Graph components}
    \label{fig:Comp}
\end{figure}

Using the terminology of Definition \ref{defi:comp}, one can rewrite the presentation of \cite{KoPi} as
\begin{teo}\label{teo:conmu}
    The group $\PAut(A_{\Gamma})$ is the group generated by all partial conjugations $C_{A}^a$ for $a\in V_\G$ and $A$ a connected component of $V_\G-\st_\G(a)$, subject to the following relations:
 \begin{itemize}
     \item[i)] $\left[C_{A}^a, C_{B}^b\right]=1$ if either $a=b$ or $a \in \lk_\G(b)$,
     \item[ii)]$\left[C_{A}^a, C_{B}^b\right]=1$ if $a \notin \st_\G(b)$ and either $A$ and $B$ are shared but distinct or $A$ or $B$ are subordinate,
     \item[iii)] $\left[C_{A}^a C_{B}^a, C_{A}^b\right]=1$ if $a \notin \st_\G(b)$, $A$ is shared and $B$ dominant.
 \end{itemize}
\end{teo}

\section{Construction of the \texorpdfstring{$\G$-} MMcCullough Miller space} 
\label{sec:construction}
We begin by defining the analog to the notion of based partitions in our case. In the case where $\G$ has no edges, so $A_\G$ is free, we can identify possible partial conjugations by a vertex $a$ with petals of based partitions of the set of vertices of $\G$ that have operative factor $a$. We will do the same for arbitrary RAAGs but then we have to take into account that the structure of the graph $\G$ imposes restrictions on the sets of allowed partitions.

\begin{defi}[Valid partitions and length]
Let $a\in\G$ be a vertex. A {\sl $\G$-valid based partition with operative factor $a$} is a partition of the form $\ua_a=\{\{a\},P_1,\ldots,P_k\}$ of the set $\G-\lk_\G(a)$ so that each $P_i$ is a union of connected components. The sets $P_i$ are called {\sl petals}. We say that $k$ is the 
    {\sl length} of the based partition and denote $k=l(\ua_a)$. The {\sl trivial} based partition with operative factor $a$ is the based partition with only one petal, i.e., $\{\{a\},\G-\st_\G(a)\}$
\end{defi}

An important observation is that, when considering $\G$-valid based partitions, the set that we are partitioning depends on the operative factor. For example, for the graph $\G$ in figure \ref{fig:gamma}, the following are $\G$-valid partitions
$$\{\{3\},\{1,4\},\{9\},\{11\},\{5,6,7\}\},$$
$$\{\{9\},\{4\},\{1,2,3,5,6,7,10,11\}\}.$$

From now on, all the based partitions will be assumed to be $\Gamma$-valid, for simplicity we will often simply say based partition.

\begin{defi}[Crossings, disjoint and compatible partition]
Crossings are defined as in the case of $\MM_n$ (Definition \ref{def:cross}), i.e., two based partitions $\ua_a$ and $\ua_b$ {\sl cross} if $a\notin \st_\G(b)$ and there are petals $P$ of $\ua_a$ and $Q$ of $\ua_b$ such that $a\not\in Q$, $b\not\in P$ and $P\cap Q\neq\emptyset$. 
Observe that the $\G$-validity condition implies that the dominant component of $\st_\G(a)$ respect to $b$ can not intersect $P$.
We define the number of crossings $\cro(\ua_a,\ub_b)$ as the number of pairs $P,Q$ as before.
We say that  $\ua_a, \ub_b$ are {\sl disjoint} if $a\notin \st_\G(b)$ and $\ua_a$, $\ub_b$ do not cross and we say that $\ua_a, \ub_b$ are {\sl compatible} if they are disjoint or $a\in\st_\G(b)$, i.e., compatibility means that either $\ua_a, \ua_b$ are disjoint, or $a$ and $b$ are adjacent, or $a=b$. A trivial based partition  is always compatible with the rest of based partitions.
\end{defi}

Looking ahead to Figure \ref{fig:RefDis}, the top of the figure represents the flower diagram of some $\G$-valid based partitions that cross where $\G$ is the graph of Figure \ref{fig:gamma}.

\begin{defi}[Vertex type]
    We define a {\sl $\G$-vertex type} or simply a {\sl vertex-type} $\uua$ as a collection of $n=|V_\G|$ pairwise compatible based partitions whose operative factors comprise all $V_\Gamma$.

    We say that two vertex types $\uua$ and $\uub$ are {\sl compatible} if their based partitions are pairwise compatible with each other. And we say that a collection of vertex types is {\sl compatible} if they are pairwise compatible. We also extend the notion of crossings in the obvious way, i.e., if $\uua$ is a vertex type and $\ub_b$ a based partition, then we set
    $$\cro(\uua,\ub_b)=\sum_{\ua_a\text{ based partition of }\uua}\cro(\ua_a,\ub_b).$$
\end{defi}

\begin{defi}[Order, Split]\label{def:order}
    Let $a\in\Gamma$ and $\ua_a$, $\underline{\widehat A}_a$ two based partitions with operative factor $a$.  We say that $\underline{\widehat A}_a\leq \ua_a$ if every petal of $\ua_a$ is contained in a petal of $\underline{\widehat A}_a$, i.e., if every petal of $\underline{\widehat A}_a$ is a union of petals of $\ua_a$. In this case, if $P$ is a petal of $\underline{\widehat A}_a$ which is a union of more than one petal of $\ua_a$, we will say that $P$ {\sl splits}. If $\underline{\widehat A}_a\leq \ua_a$, then $ l(\underline{\widehat A}_a) \leq l(\ua_a) $. 

    For vertex types, we will say that $\underline{\ua}\leq \underline{\ub}$ if $\ua_a\leq \ub_a$, for any $a\in V_\Gamma$.
\end{defi}

The fact that combining petals can not increase the number of crossings implies the following.

\begin{lem}\label{lem:ordercomp} Let $\ua_a<\ub_a$, $\us_c$ be based partitions and $\uua<\uub$, $\underline{\us}$ vertex types. Then
\begin{itemize}
\item[i)] $\cro(\ua_a,\us_c)\leq\cro(\ub_a,\us_c)$. As a consequence, if $\ub_a$ is compatible with $\us_c$, then so is $\ua_a$. 

\item[ii)] $\cro(\uua,\us_c)\leq\cro(\uub,\us_c)$. As a consequence, if $\uub$ is compatible with a vertex type $\underline{\us}$ then so is $\uua$. In particular, $\uua$ and $\uub$ are compatible.
\end{itemize}
\end{lem}

\begin{defi}[$\Gamma$-Whitehead poset, nuclear vertex]
    The $\Gamma$-Whitehead poset $\Wh_{\G}$ is the poset formed by the $\G$-vertex types under the partial order previously defined. The poset $\Wh_{\G}$ has a unique minimal element called nuclear vertex $\OO$, whose based partitions for each $a\in V_\G$ are all trivial, i.e., have only one petal.
\end{defi}

\begin{figure}
   \begin{tikzpicture}
       \centering

    \node[shape=circle,draw=black, fill=black, label={above: 2}, scale=0.5] (1) at (-6,-1.5) {};
    \node[shape=circle,draw=black, fill=black, label={above: 3}, scale=0.5] (2) at (-5,-1.5) {};
    \node[shape=circle,draw=black, fill=black, label={above: 4}, scale=0.5] (3) at (-4,-1.5) {};
    \node[shape=circle,draw=black, fill=black, label={above: 5}, scale=0.5] (4) at (-3,-1.5) {};
     \node[shape=circle,draw=black, fill=black, label={above: 6}, scale=0.5] (5) at (-2,-1.5) {};
    \node[shape=circle,draw=black, fill=black, label={above left: 8}, scale=0.5] (6) at (-4,-0.5) {};
    \node[shape=circle,draw=black, fill=black, label={above left: 9}, scale=0.5] (7) at (-4,0.5) {};
    \node[shape=circle,draw=black, fill=black, label={ above: 7}, scale=0.5] (8) at (-1,-1.5) {};
   
    \node[shape=circle,draw=black, fill=black, label={below left: 10}, scale=0.5] (9) at (-4,-2.5) {};
    \node[shape=circle,draw=black, fill=black, label={below left: 11}, scale=0.5] (10) at (-4,-3.5) {};
    \node[shape=circle,draw=black, fill=black, label={ above: 1}, scale=0.5] (11) at (-7,-1.5) {};

    \path (2) edge (1);
   \path (2) edge (9);
    \path (2) edge (6);
   
    \path (4) edge (5);
    \path (4) edge (9);
    \path (4) edge (6);
    \path (7) edge (6);
    \path (5) edge (8);
    \path (10) edge (9);
    \path (1) edge (11);
 
   \end{tikzpicture}
   
    \caption{A graph $\G$ such that the based partitions of Figure \ref{fig:RefDis} are $\G$-valid} \label{fig:gamma}
\end{figure}

As we did for pure symmetric automorphisms of free groups, we are going to add \lq\lq markings'' to the elements of $\Wh_\G$ to define a new poset, the $\G$-McCullough-Miller poset, which admits an action of the group $\PAut(A_\G)$.

\begin{defi}[Marked vertex type and the action of $\PAut(A_\G)$] We say that a generating system $X$ of $A_\G$ is a {\sl basis} if there is some $\alpha\in\PAut(A_\G)$ such that $X=\alpha(V_\G)$.
A {\sl marked based partition} is a pair $(X,\ua)$ where $X$ is an ordered basis of $A_\G$ and $\ua$ is a based partition and a {\sl marked vertex type} $(X,\uua)$ is the family of marked based partitions $(X,\ua_a)$ where $\ua_a$ runs over the set of based partitions of a vertex type $\uua$. As in Definition \ref{def:marking}, by marking a based partition $\ua$ we understand that we are relabeling the elements in the partition using $X$ (via the automorphism $\alpha$). We extend the partial order of Definition \ref{def:order} to marked vertex types as follows: $(X_1,\uua)\leq(X_2,\uub)$ if $X_1=X_2$ and $\uua\leq\uub$. The marked vertex types form a poset and the group $\PAut(A_\G)$ acts on this poset via $$\alpha(X,\uua)=(\alpha(X),\uua)$$
where $\alpha\in\PAut(A_\G)$.
\end{defi}

\begin{defi}[Carried automorphism]
    Let $\alpha\in\PAut(\A)$, we say that $\alpha$ is carried by a based partition $\ua_a$ with operative factor $a$ if
     \begin{itemize}
         \item[i)] $\alpha(a)=a$, i.e. the operative factor gets fixed.
         \item[ii)] for any $b\in V_\G-\st_\G(a)$, $\alpha(b)$ is the result of conjugating $b$ by a power of $a$.
         \item[iii)] If $b,c$ lie in the same petal of $\ua_a$, then they get conjugated by the same power of $a$ when applying $\alpha$.
     \end{itemize}
     We say that $\ua_a$ is a \textit{full carrier} of $\alpha$ if it carries $\alpha$ and each petal gets conjugated by a different power of $a$.  If $\alpha$ is not the identity and is carried by some based partition, then a full carrier is uniquely determined,  to see it note that it is enough to take any based partition $\ua_a$ that carries $\alpha$ and then join all the petals where the action is the same. 
     
     We say that $\alpha$ is {\sl carried by a vertex type} $\underline{\ua}$ if $\alpha=\alpha_1\cdots\alpha_k$ where each $\alpha_i$ is carried by some based partition $\ua_a$ of $\underline{\ua}$.

     Finally, we say that $\alpha$ is {\sl carried by a marked vertex type} $(X,\uua)$ if $\gamma\alpha\gamma^{-1}$ is carried by $\uua$ where $\gamma(X)=V_\G$.
\end{defi}
Observe that an automorphism $\alpha$ is carried by a based partition $\ua_a$ if and only if it lies in the subgroup generated by the partial conjugations $C^a_B$ where $B$ runs over the petals of $\ua_a$.  Note also that the full carrier of a partial conjugation $C^a_A$ is the based partition
$$\{\{a\}, A, \G-(\st_\G(a)\cup A)\}.$$
The following is an easy observation that will be useful below.

\begin{lem}\label{lem:ordercarrier} The set of automorphisms carried by a vertex type is a subgroup of $\PAut(A_\G)$ that contains the inner automorphisms. Moreover,
given two vertex types $\uua$ and $\uub$ such that $\uua<\uub$, the subgroup of automorphisms carried by $\uua$ is inside the subgroup of automorphisms carried by $\uub$.
\end{lem}

In what follows, we will often consider the subgroups of outer automorphisms carried by either a given based partition or a vertex type, which are just the corresponding quotients over $\Inn(A_\G)$.

Next, we define an equivalence relationship in the set of marked vertex types as follows.

\begin{defi}[Vertex types equivalence relation]
    We say that two marked vertex types $(X_1,\underline{\ua})$ and $(X_2,\underline{\ua})$ are equivalent and we put $(X_1,\underline{\ua})\equiv (X_2,\underline{\ua})$ if
there exists an $\alpha\in\PAut(\A)$ carried by $(X_1,\uua)$ such that $\alpha(X_1)=X_2$. We  denote by $[X_1,\underline{\ua}]$ the equivalence class of $(X_1,\underline{\ua})$ under this relation. \label{def:EqRel}
\end{defi} 

In particular, observe that if $\alpha$ is inner, then $[X_1,\underline{\ua}]=[\alpha(X_1),\underline{\ua}]$. The same observation as after Definition \ref{def:carriedFree} applies and we have a well defined poset structure in the set of equivalence classes of marked vertex types.

\begin{defi}[$\G$-McCullough-Miller poset]
    The $\G$-McCullough-Miller poset is the poset formed by the set of equivalence classes of marked vertex types under the equivalence relationship $\equiv$ of Definition \ref{def:EqRel} with the induced partial order.

    The group $\PAut(A_\G)$ acts on this poset via $$\alpha([X,\uua])=[\alpha(X),\uua]$$
    and this action factors through the epimorphism $\PAut(A_\G)\to\POut(A_\G)$ so we also have an action of $\POut(A_\G)$.
\end{defi}

The $\G$-McCullough-Miller poset is formed by copies of the $\Gamma$-Whitehead poset attached to each other through the equivalence relation $\equiv$ defined before. It is easy to verify that  nuclear vertex types only carry inner automorphisms, so there is a bijection between the set of nuclear vertices $[X,\OO]$ in the $\G$-McCullough-Miller poset and the set of conjugacy classes of bases of $A_\G$.

\begin{defi}[$\MM_\Gamma$ complex]\label{def:MMG}
    We define the McCullough-Miller complex $\MM_\G$ for a RAAG $A_\Gamma$ as the simplicial realization of the $\G$-McCullough-Miller poset. 
\end{defi}

\begin{lem}\label{lem:intersect}
    The stars of the nuclear vertices $[X_1,\OO]$ and $[X_2,\OO]$ intersect if and only if there is an $\alpha\in\PAut(\A)$ carried by some $(X_1,\underline{\ua})$ such that $\alpha(X_1)=X_2$.
\end{lem}

\begin{proof}
    The stars intersect if and only if we can represent a vertex both as $[X_1,\uua]$ and $[X_2,\uua]$. By definition of the equivalence relation (\ref{def:EqRel}) we have that this can only happen if there is an $\alpha\in\PAut(\A)$ carried by $(X_1,\underline{\ua})$ such that $\alpha(X_1)=X_2$.
\end{proof}

\begin{nota}\label{rem:cosetposet} The definition of the action above implies that the stabilizer $H=\Stab_{\POut(A_\G)}[V_\G,\underline{\ua}]$  of the vertex $[V_\G,\underline{\ua}]$ in $\POut(A_\G)$ is precisely the subgroup generated by those outer automorphisms carried by the vertex type $\underline{\ua}$. Therefore if $\alpha$ is such that $X=\alpha(V_\G)$, the stabilizer of the vertex $[X,\underline{\ua}]=\alpha([V_\G,\uua])$ is $\alpha H\alpha^{-1}$. Taking this into account, one sees that the McCullough-Miller complex $\MM_\G$ can be also defined as follows. Recall that for a poset $\mathcal{F}$ of subgroups of a group $G$, one can define a new poset, called the {\sl coset poset}, as
$$C(\mathcal{F})=\{gH\mid H\in\mathcal{F},g\in G\}.$$
Then, the McCullough-Miller poset is the coset poset of the family 
$$\{\Stab_{\POut(A_\G)}[V_\G,\underline{\ua}]\mid \uua\text{ vertex type}\}$$
of subgroups of $\POut(A_\G)$ and $\MM_\G$ is the associated geometric realization.
\end{nota}

We will see below that these stabilizers are free abelian groups, that their rank can be computed in terms of the length of the based partitions of $\uua$ and that they yield in fact all the cell stabilizers of the McCullough-Miller complex $\MM_\G$.

\section{Compatibility, refinement, disjunction}\label{sec:notation}

\subsection{Compatibility and stabilizers}

We begin this section introducing some notation that will be useful to understand when two based partitions are compatible. 
Let $\ua_a$ be a based partition. For each $b\in V_\G$ not adjacent to $a$ we denote by $D^b_{\ua_a}$ the petal of $\ua_a$ that contains $b$.
\


\begin{lem}\label{lem:compatible} Let $a,b\in V_\G$ and $\ua_a$ and $\ub_b$ based partitions with $a\not\in\st_\G(b)$. Set
$$\ua_a=\{\{a\},D^b_{\ua_a},P_1,\ldots,P_t\},$$
$$\ub_b=\{\{b\},D^a_{\ub_b},Q_1,\ldots,Q_k\}.$$
Then each of the petals $P_i$ is a union of shared and/or subordinate components of $\G-\st_\G(a)$ respect to $b$, whereas $D^b_{\ua_a}$ contains the dominant component and possibly some of the shared and/or subordinate components.  If there are crossings, say $P_i\cap Q_j\neq\emptyset$, then $P_i\cap Q_j$ is a union of shared components only and the analogous statements hold for $b$.

Moreover, $\ua_a$ and $\ub_b$ are compatible if and only if
$$\begin{aligned}
D^a_{\ub_b}&=P_1\cup\cdots\cup P_t\cup(\st_\G(a)-\lk_\G(b))\text{ and }\\
D^b_{\ub_a}&=Q_1\cup\cdots\cup Q_k\cup(\st_\G(b)-\lk_\G(a)).\\
\end{aligned}$$
\end{lem}
\begin{proof} As $\ua_a$ and $\ub_b$ are $\G$-valid, petals must be unions of connected components of the complements of the corresponding stars. Obviously, for, say, $a$, the dominant component must be in $D^b_{\ua_a}$ and then the $\G$-validity condition implies $\st_\G(b)-\lk_\G(a)\subseteq D^b_{\ua_a}$. The statement about the $P_i$'s is obvious, and if $P_i\cap Q_j\neq\emptyset$, this intersection can not contain any subordinate component, so it can only consist of shared ones. 

Assume now that $\ua_a$ and $\ub_b$ are compatible. As $a$ and $b$ are not adjacent, $\ua_a$ and $\ub_b$ must be disjoint. So there are no crossings, meaning that the union $P_1\cup\cdots\cup P_t$ must lie inside $\st_\G(b)\cup D^a_{\ub_b}$. But as the $P_i$'s are unions of shared and/or subordinate components, they can not intersect $\st_\G(b)$ so $P_1\cup\cdots\cup P_t\subseteq D^a_{\ub_b}$. This implies 
$$D^a_{\ub_b}=P_1\cup\cdots\cup P_t\cup(\st_\G(a)-\lk_\G(b))$$
    and by symmetry we have the same for $b$. The converse is obvious. 
\end{proof}

\begin{nota}
    Observe that if two based partitions $\ua_a$ and $\ub_b$ cross, then there must be some shared component between $\G-\st_\G(a)$ and $\G-\st_\G(b)$. When this happens, the vertices $a$ and $b$ are said to form a {\sl separating intersection link, or a SIL pair} (see for example \cite{DayWade}).
\end{nota}

We have seen already that the different based partitions associated to a given vertex type must be compatible. This has very strong consequences for the sets of outer automorphisms carried by them, as we shall now see.

\begin{lem}\label{lem:compatible1} Let $\alpha=C_A^a$, $\beta=C_B^b$ be partial conjugations for $a,b\in V_\G$ carried by compatible based partitions $\ua_a$ and $\ub_b$. Then $\alpha$ and $\beta$ commute modulo $\Inn(A_\G)$ and 
either $a\in\st_\G(b)$ or, by multiplying by inner automorphisms if necessary, we may transform the partial conjugations $\alpha$ and $\beta$ so that we may assume $A\cap B=A\cap\st_\G(b)=B\cap\st_\G(a)=\emptyset$. 
\end{lem}
\begin{proof} If $a$ and $b$ are either equal or adjacent we are in case $i)$ of Theorem \ref{teo:conmu} and $\alpha$ and $\beta$ commute. So from now on we assume that $a\not\in\st_\G(b)$.

Put 
$\ua_a=\{\{a\},D^b_{\ua_a},P_1,\ldots,P_t\}$ and $\ub_b=\{\{b\},D^a_{\ub_b},Q_1,\ldots,Q_k\}$ where, as before, $D^b_{\ua_a}$ and $D^a_{\ub_b}$ are the petals that contain the respective dominant components.

As $C_A^a$ is carried by $\ua_a$, $A$ is a union of (some of) the petals in $\ua_a$. If $D^b_{\ua_a}$ is one of those, then multiplying by the inner automorphism $C_{V_\G}^{a^{-1}}$ we get another representative of the same outer automorphism as $C_A^a$ that acts by conjugation with $a^{-1}$ and for which the corresponding $A$ does not contain the petal $D^b_{\ua_a}$. So, up to multiplying by this inner automorphism we may assume that $A\subseteq P_1\cup\cdots\cup P_{t}$.  The same argument for $\beta$ implies that we may also assume $B\subseteq Q_1\cup\cdots\cup Q_{k}$. This and the fact that there are no crossings implies that decomposing $A$ and $B$ into connected components we may factor $C_A^a$ and $\beta$ as product of partial conjugations which are in case $ii)$ of Theorem \ref{teo:conmu}, so $C_A^a$ and $\beta$ commute. 

Moreover, the conditions above imply $A\cap B=\st_\G(b)\cap A=\st_\G(a)\cap B=\emptyset$ (recall that the petals $P_i$ do not intersect $\st_\G(b)$ and the same for the $Q_j$'s and $\st_\G(a)$).
\end{proof}

As a consequence, we have:

\begin{lem}\label{lem:freeabelian} Let $\uua$ be a vertex type. Then the  outer automorphisms carried by $\uua$ form a free abelian group. This group is precisely the stabilizer of the vertex $v=[V_\G,\uua]$ of $\MM_\G$ under the action of $\POut(A_\G)$.

\end{lem}
\begin{proof} 
The outer automorphisms carried by $\uua$ form a subgroup of $\POut(A_\G)$ generated by the images of the  partial conjugations carried by the based partitions of $\uua$. These can be written as products of partial conjugations $C^a_P$ where $P$ is a petal of a based partition $\ua_a$ of $\uua$. Therefore it is finitely generated. It is abelian by Lemma \ref{lem:compatible1} and torsion free because the ambient group $\POut(A_\G)$ is, so it is free abelian. For the last statement, see the comment after Lemma \ref{lem:intersect}.
\end{proof}

We can now deduce Theorem A:

\medskip

\noindent{\sl Proof of Theorem A.} We identify the simplicial realization of the $\G$-Whitehead poset with the subcomplex induced by vertices of the form $[V_\G,\uua]$. By construction,  this is a strong fundamental domain for the action of $\POut(A_\G)$ on $\MM_\G$. For the statement about stabilizers, note that we only have to consider stabilizers of cells in the strong fundamental domain. We have seen that the stabilizer of the vertex $[V_\G,\underline{\ua}]$ in $\POut(A_\G)$ is precisely the subgroup generated by those outer automorphisms carried by the vertex type $\underline{\ua}$, moreover Lemma \ref{lem:ordercarrier} together with Remark \ref{rem:cosetposet} imply that the stabilizer of the simplex
$$[V_\G,\uua^0]<[V_\G,\uua^1]<\cdots<[V_\G,\uua^k]$$
is the stabilizer of $[V_\G,\uua^0]$, in other words,
we only have to consider vertex stabilizers so it is enough to use Lemma \ref{lem:freeabelian}. \qed

In fact, the condition that the based partitions are compatible is not only sufficient, it is necessary for the carried outer automorphisms to commute as we shall see now.

\begin{lem}\label{lem:compatible2}
    Two based partitions $\ua,\ub$ are compatible if and only if the outer automorphisms carried by $\ua$ commute with the outer automorphisms carried by $\ub$. \label{lem:compCarac}
\end{lem}
\begin{proof} We have seen in Lemma \ref{lem:compatible1} that compatibility implies that the carried outer automorphisms commute. 
The converse is obvious if $a$ and $b$ are either equal or adjacent so we assume this is not the case. Put 
$\ua=\{\{a\},D^b_{\ua_a},P_1,\ldots,P_t\}$ and $\ub=\{\{b\},D^a_{\ua_b},Q_1,\ldots,Q_k\}$ where $D^b_{\ua_a}$ and $D^a_{\ua_b}$ contain the respective dominant components so $a\in D^a_{\ua_b}$ and $b\in D^b_{\ua_a}$.  Assume for a contradiction that $\ua$ and $\ub$ are not disjoint so there is at least a cross between two petals $P_i, Q_j$ and take $c\in P_i\cap Q_j$. Consider the partial conjugations $\alpha=C^a_{P_i}$ and  $ \beta=C^b_{Q_j}$ and assume that there is some $g\in A_\G$ such that for the internal automorphism $\gamma\in\Inn(A_\G)$ that consists of conjugation by $g^{-1}$ we have $$\alpha\beta=\gamma\beta\alpha.$$
Then
    $$\begin{aligned}
\alpha\beta(a)=a=gag^{-1}&=
\gamma\beta\alpha(a),\\
\alpha\beta(b)=b=gbg^{-1}&=
\gamma\beta\alpha(b),\\
\alpha\beta(c)=abcb^{-1}a^{-1}&=gbaca^{-1}b^{-1}g^{-1}=
\gamma\beta\alpha(c).\\
    \end{aligned}$$
So $g$ lies in the centralizers of $a$ and of $b$ in $A_\G$. These centralizers are generated by $\st_\G(a)$ and $\st_\G(b)$ respectively and as $b\not\in\st_\G(a)$ and $a\not\in\st_\G(b)$, we deduce that there is no $a$ nor $b$ in any graphically reduced expression of $g$. From the last equation above we have
$$abc^{-1}b^{-1}a^{-1}gbaca^{-1}b^{-1}=g$$
and we see that if $g$ is graphically reduced, so is the left hand expression so we have a contradiction.
    \end{proof}

This result gives another proof of the fact (Lemma \ref{lem:ordercomp}) that for vertex types $\uua$ and $\uub$ with $\uua<\uub$, the corresponding based partitions are compatible, because  Lemma \ref{lem:ordercarrier} implies that the subgroup of carried automorphisms of $\uua$ is inside the subgroup of carried automorphisms of $\uub$.

\begin{defi}[Join]
    If a collection of vertex types $\{\underline{\ua}^1,\ldots,\underline{\ua}^k\}$ has a least upper bound, we denote it by $\underline{\ua}^1\vee\cdots\vee \underline{\ua}^k$, and call it the {\sl join} of $\{\underline{\ua}^1,\ldots,\underline{\ua}^k\}$. 
\end{defi}

\begin{lem}
    A collection $\Omega=\{\underline{\ua}^1,\ldots,\underline{\ua}^k\}$ of vertex types is pairwise compatible if and only if it has an upper bound. Moreover, if it has an upper bound, then it has a least upper bound. \label{lem:ComSup}
\end{lem}
\begin{proof} Assume first that there is an upper bound. Then Lemma \ref{lem:ordercarrier} implies that the subgroup of carried outer automorphisms of each vertex type in $\Omega$ is inside that of the upper bound. That subgroup is free abelian by Lemma \ref{lem:freeabelian} so those carried outer automorphisms commute with each other and by Lemma \ref{lem:compatible2}, the based partitions of the vertex types in $\Omega$ must be all compatible. It is also easy to provide an alternative direct argument: if there is an upper bound, Lemma \ref{lem:ordercomp} implies that the vertex types in $\Omega$ are all compatible with the upper bound. Suppose that $\underline{\ua}$ and $\underline{\ub}$ are vertices in $\Omega$ that have based partitions which are not compatible with each other. In such a case, if we make divisions of the petals of $\underline{\ub}$, the partitions will remain not compatible (dividing petals can not reduce the number of  crossings). However, this leads to a contradiction because by dividing petals in $\underline{\ub}$ we could reach the upper bound, which is compatible with $\underline{\ua}$.

We will show the converse constructively: if $\ua^1_{a}$ and $\ua^2_{a}$ are based partitions with the same operative factor $a$, we define their union $\ua^1_{a} \vee \ua^2_{a}$ as the based partition with the same operative factor $a$ and whose petals are the non-empty intersections of the petals of $\ua^1_{a}$ and $\ua^2_{a}$. More generally, given a collection $\{\ua^1_{a}, \ldots, \ua^m_{a}\}$ of based partitions with the same operative factor, we define $\ua^1_{a} \vee \cdots \vee \ua^m_{a}$ inductively as $((\ua^1_{a} \vee \ua^2_{a}) \vee \ua^3_{a}) \cdots \vee \ua^m_{a}$. Observe that this process keeps the $\G$-validity (intersections of petals are unions of shared components).

For a collection of compatible vertex types $\{\underline{\ua}^1,\ldots,\underline{\ua}^k\}$, we define $\underline{\ua}^1 \vee \cdots \vee \underline{\ua}^k$ as the vertex type whose non-trivial based partition with operative factor $a$ is the union of all the non-trivial based partitions of the family with operative factor $a$. Recall that to ensure that this union is a vertex type, i.e., that its based partitions are pairwise compatible, we need the initial collection to be compatible.  It is obvious that $\underline{\ua}^1 \vee \cdots \vee \underline{\ua}^k$ is an upper bound and is easy to check using the definitions that any other upper bound must be greater or equal than the join. 
Intuitively, when joining two based partitions we are computing their ``least common multiple''.

\end{proof}

\subsection{Refinement, disjunction}

In this subsection we introduce, following \cite{MM}, a few technical notions that  will be needed in the proof that $\MM_\G$ is contractible. Although for some results one can refer to \cite{MM}, we have tried to include full proofs and to provide more details to clarify the arguments, moreover, we need to check that everything works fine under our new $\G$-validity condition. 

\begin{defi}[Refinement]
    Let $\ua_a$ and $\ub_b$ be  based partitions with operative factors $a$ and $b$.  The refinement $\ua'_a$ of $\ua_a$ with respect to $\ub_b$ is defined as follows: if $\ua_a$ and $\ub_b$ are compatible, then $\ua'_a=\ua_a$ and in other case $\ua'_a$ is the based partition with operative factor $a$ and petals described next. The petal containing $b$  remains unchanged, i.e., $D^b_{\ua_a}=D^b_{\ua_a'}$; whereas the remaining petals are the result of splitting the rest of petals of $\ua_a$ in subpetals according to the crossings with $\ub_b$ in the following way. The elements in each of the non empty intersections $P\cap Q$ for $D^b_{\ua_a}\neq P$ petal of $\ua_a$ and $D^a_{\ub_b}\neq Q$ petal of $\ub_b$ form  a new petal in $\ua_a'$ and there is possibly one more new petal for each such $P$ with all the elements of $P$ not involved in any crossing (see Figure (\ref{fig:RefDis})). 
\end{defi}

    This definition is exactly as in \cite{MM} but in our case
    we must check that the refinement of $\ua_a$ with respect to $\ub_b$ is a $\G$-valid partition if $\ua_a$ and $\ub_b$ are. To do that, let $\ua'_a$ be the refinement of $\ua_a$ with respect to $\ub_b$, we may assume that $\ua_a$ and $\ub_b$ are not compatible and in particular $a$ and $b$ are not adjacent. Suppose that $\ua'_a$ is not $\G$-valid, this means that there exists a petal $P\in \ua_a$ which crosses a petal $Q\in \ub_b$ such that there are at least two vertices $v,w\in P$ which are connected in $\Gamma-\st_\G(a)$ and only one them is in $Q$. Then there is a path connecting $v$ and $w$ lying in $\Gamma-\st_\G(a)$ but $v$ and $w$ are not connected in $\Gamma-\st_\G(b)$, so this path meets $\st_\G(b)$ and therefore connects the three vertices $v,w,b$ in $\Gamma-\st_\G(a)$. As $\ua_a$ is $\G$-valid, this implies $b\in P$, so $P$ contains the dominant component, i.e. $P=D^b_{\ua_a}$. This
contradicts one of the requirements for $P\cap Q$ to be a crossing.

\begin{defi}[Disjunction]
     Let $\ua_a$ and $\ub_b$ be based partitions with operative factors $a,b$ such that $a\notin \st_\G(b)$. 
     Let $\ua'_a$ be the refinement of $\ua_a$ with respect to $\ub_b$. We define the disjunction $\ua''_a$ of $\ua_a$ with respect to $\ub_b$ as the based partition with operative factor $a$ and the following petals: those petals of $\ua_a$ which do not cross $\ub_b$ remain unchanged, and the rest of petals are all combined with $D^b_{\ua_a}$.
\end{defi}

\begin{figure} 
\begin{tikzpicture}
   \centering
    \draw (10,4) node[minimum size=1cm,circle,draw] (O) {3};
       \draw (O.80) to[out=50,in=-10,looseness=10]  node[pos=0.1](C1){} node[pos=0.4](C2){} node[pos=0.8](C3){} node[pos=0.9](C4){}  (O.-40);
\node[fit=(C1)(C2)(C3)(C4)]{$1$\quad \quad$4$};      
 \draw (O.120) to[out=135,in=210,looseness=8]  node[pos=0.1](D1){} node[pos=0.4](D2){} node[pos=0.8](D3){} node[pos=0.9](D4){}  (O.-110);
\node[fit=(D1)(D2)(D3)(D4)]{$11$};      
\draw (10,3) node {$\ua_3$};
 \draw (13.8,4) node[minimum size=1cm,circle,draw] (P) {5};
       \draw (P.80) to[out=50,in=-10,looseness=10]  node[pos=0.1](E1){} node[pos=0.4](E2){} node[pos=0.8](E3){} node[pos=0.9](E4){}  (P.-40);
\node[fit=(E1)(E2)(E3)(E4)]{$9$};      
 \draw (P.120) to[out=135,in=195,looseness=12]  node[pos=0.0](F1){} node[pos=0.1](F2){} node[pos=0.85](F3){} node[pos=0.9](F4){}  (P.-150);
\node[fit=(F1)(F2)(F3)(F4)]{$7$};   
\draw (13.8,3) node {$\ub_5$};

\draw (11.9,6) node {2 \quad \quad 6};

   \draw (11.9,2) node {2 \quad \quad 6};
        \draw (10,0) node[minimum size=1cm,circle,draw] (O) {3};
       \draw (O.100) to[out=80,in=30,looseness=10]  node[pos=0.1](C1){} node[pos=0.4](C2){} node[pos=0.8](C3){} node[pos=0.9](C4){}  (O.40);
\node[fit=(C1)(C2)(C3)(C4)]{$1$};      

 \draw (O.35) to[out=35,in=-10,looseness=12]  node[pos=0.1](C1){} node[pos=0.4](C2){} node[pos=0.8](C3){} node[pos=0.9](C4){}  (O.-40);

 \draw (O.120) to[out=135,in=210,looseness=8]  node[pos=0.1](D1){} node[pos=0.4](D2){} node[pos=0.8](D3){} node[pos=0.9](D4){}  (O.-110);
\node[fit=(D1)(D2)(D3)(D4)]{$11$};      
\draw (10,-1) node {$\ua'_3$};
 \draw (13.8,0) node[minimum size=1cm,circle,draw] (P) {5};
       \draw (P.80) to[out=50,in=-10,looseness=10]  node[pos=0.1](E1){} node[pos=0.4](E2){} node[pos=0.8](E3){} node[pos=0.9](E4){}  (P.-40);
\node[fit=(E1)(E2)(E3)(E4)]{$9$};      
 \draw (P.120) to[out=135,in=195,looseness=12]  node[pos=0.0](F1){} node[pos=0.1](F2){} node[pos=0.85](F3){} node[pos=0.9](F4){} node[pos=0.5](F5){} node[pos=0.8](F6){} (P.-150);
\node[fit=(F1)(F2)(F3)(F4)]{$7$};   
\node[fit=(F5)(F6)]{$4$}; 
\draw (13.8,-1) node {$\ub_5$};

\draw (11.9,-2.3) node {2 \quad \quad 6};
  \draw (10,-4) node[minimum size=1cm,circle,draw] (O) {3};
       \draw (O.100) to[out=80,in=30,looseness=10]  node[pos=0.1](C1){} node[pos=0.4](C2){} node[pos=0.8](C3){} node[pos=0.9](C4){}  (O.40);
\node[fit=(C1)(C2)(C3)(C4)]{$1$};

 \draw (O.120) to[out=135,in=210,looseness=8]  node[pos=0.1](D1){} node[pos=0.4](D2){} node[pos=0.8](D3){} node[pos=0.9](D4){}  (O.-110);
\node[fit=(D1)(D2)(D3)(D4)]{$11$};      
\draw (10,-5) node {$\ua''_3$};
 \draw (13.8,-4) node[minimum size=1cm,circle,draw] (P) {5};
       \draw (P.80) to[out=50,in=-10,looseness=10]  node[pos=0.1](E1){} node[pos=0.4](E2){} node[pos=0.8](E3){} node[pos=0.9](E4){}  (P.-40);
\node[fit=(E1)(E2)(E3)(E4)]{$9$};      
 \draw (P.120) to[out=135,in=195,looseness=12]  node[pos=0.0](F1){} node[pos=0.1](F2){} node[pos=0.85](F3){} node[pos=0.9](F4){} node[pos=0.5](F5){} node[pos=0.8](F6){} (P.-150);
\node[fit=(F1)(F2)(F3)(F4)]{$7$};   
\node[fit=(F5)(F6)]{$4$}; 
\draw (13.8,-5) node {$\ub_5$};
    \end{tikzpicture}
     \caption{Refinement and disjunction of $\ua_3$ respect to $\ub_5$}\label{fig:RefDis}
    \end{figure}

\begin{ej}
    We can see an example of refinement and disjunction in Figure \ref{fig:RefDis} for $\G$-valid based partitions for the graph $\G$ of Figure \ref{fig:gamma}.   
    $$\ua_3=\{\{3\},\{11\},\{1,4\},\{5,6,7,9\}\}\text{ and}$$
    $$\ub_5=\{\{5\},\{4,7\},\{9\},\{1,2,3,11\}\}.$$ The based partitions $\ua_3$ and $\ub_5$ cross precisely at the petals $\{1,4\}\in\ua_3$ and $\{4,7\}\in\ub_5$; recall that intersections with the petals containing the dominant components, $D_{\ua_3}^5=\{5,6,7,9\}$ and $D_{\ub_5}^3=\{1,2,3,11\}$, are not crossings. The refinement and disjunction of $\ua_3$ respect to $\ub_5$ are 
    $$\ua_3'=\{\{3\},\{11\},\{1\},\{4\},\{5,6,7,9\}\}\text{ and }$$ 
    $$\ua_3''=\{\{3\},\{11\},\{1\},\{4,5,6,7,9\}\}.$$ Recall that we do not draw $D_{\ua_3}^5$ nor $D_{\ub_5}^3$  in the flower diagrams, in the terminology of McCullough-Miller, this petals are called \textit{infinity components}.
\end{ej}

If $\ua_a$ and $\ub_b$ are $\G$-valid then
$\ua''_a$ is obviously $\G$-valid as it is obtained from $\ua'_a$ by collapsing petals, and collapsing petals keeps the $\G$-valid property (elements that were together remain together). 

The next properties follow straightforward from the definitions:
\begin{itemize}
    \item[i)] the operative factors of $\ua_a$, $\ua'_a$ and $\ua''_a$ are the same,
    \item[ii)] $\cro(\ua_a,\ub_b)=\cro(\ua'_a,\ub_b)$ and $\cro(\ua''_a,\ub_b)=0$,
    \item[iii)] $\ua_a\leq \ua'_a$ and $\ua_a''\leq \ua_a'$, and
    \item[iv)] if $\ua_a$ and $\ub_b$ are disjoint, then $\ua_a=\ua_a'=\ua_a''$. This always occurs if either $ \ua_a$ or $\ub_b$ is the trivial based partition.
\end{itemize}

\begin{lem}
    Let $\ua_a, \ub_b$ be two not compatible based partitions and let $\ua_a',\ua_a''$ be the refinement and disjunction of $\ua_a$ with respect to $\ub_b$. If $\ua_a, \ub_b$ are compatible with a collection of based partitions then so are $\ua_a',\ua_a''$. \label{lem:collComp}
\end{lem}
\begin{proof}
    It is enough to prove that if $\ua_a$ and $\ub_b$ are compatible with a based partition $\us_c$, then so is the refinement $\ua_a'$ of $\ua_a$ with respect to $\ub_b$. Combining petals we can not generate new crossings, so the disjunction will not cause any problem. 

As always,  $a$, $b$ and $c$ are the operative factors of $\ua_a$, $\ub_b$ and $\us_c$ respectively. If $a$ is equal or adjacent to $c$ there is nothing to prove. Now, we distinguish two cases. Assume first that $b$ is not equal or adjacent to $c$. Then the result follows exactly with the same proof as in Lemma 3.4 of \cite{MM} so we are left with the case when $a\not\in\st_\G(c)$, $b\in\st_\G(c)$. This implies that $b$ and $c$ are in the same connected component of $\Gamma - \st_\G(a)$.
    Assume that $\ua_a'$ is not compatible with $\us_c$. As $\ua_a$ was, this means that the petal of $\ua_a$ which contains $c$ has been divided when refining with respect to $\ub_b$; but this is a contradiction, since $b$ is in that petal too, and in consequence crossings can not involve that petal.
\end{proof}

 \begin{lem}\label{lem:charcr} Let  $\ua_a,\underline{\widehat A}_a,\us_c$ be based partitions such that $\ua_a<\underline{\widehat A}_a$ and $a\not\in\st_\G(c)$.  
 If $\ua_a$ and $\us_c$ are compatible but $\underline{\widehat A}_a$ crosses $\us_c$ then $D^c_{\ua_a}$ is split in $\underline{\widehat A}_a$.
\end{lem}
\begin{proof} Set $\ua_a=\{\{a\},D^c_{\ua_a},P_1,\ldots,P_t\}$ and $\us_c=\{\{c\},D^a_{\us_c},Q_1,\ldots,Q_k\}$. Since they are compatible $P_i\cap Q_j=\emptyset$, for any $i,j$. Now, assume that $D^c_{\ua_a}$ is not split in $\underline{\widehat A}_a$. Then, one of the $P_i$'s has been split into $P_{i_1},...,P_{i_k}$, and since $\underline{\widehat A}_a$ crosses $\us_c$ we have that at least one of those new petals, say $P_{i_1}$ crosses at least one $Q_j$, but this is a contradiction since $P_{i_1}\subset P_i$ and $P_i\cap Q_j=\emptyset$.
\end{proof}

The previous lemma is based on the fact that given two compatible based partitions, the only way for crossings to appear is by splitting at least one of the petals containing the dominant component. Note that the reciprocal is not true, i.e. we could split the petal containing the dominant component and the based partitions could still be compatible, as we can see in the next example:
\begin{ej} Let $\G$ consist of 7 isolated vertices with labels from 1 to 7.
    Take 
    $$\ua_1=\{\{1\},\{2,3\},\{4,5,6,7\}\},$$ 
    $$\us_4=\{\{4\},\{1,2,3,7\},\{5,6\}\},\text{ and}$$ 
    $$\underline{\widehat A}_1=\{\{1\},\{2,3\},\{4,5,6\}, \{7\}\}.$$ 
   Then $\ua_1\leq\underline{\widehat A}_1$, the petal $D^4_{\ua_1}=\{4,5,6,7\}$ has been divided in $\underline{\widehat A}_1$ and both $\ua_1$ and $\underline{\widehat A}_1$ are compatible with $\us_4$. 
\end{ej}


\begin{defi}[$c$-contained]
   Let $c\in\G$ and $\ua_a$, $\ua_b$ two based partitions of a vertex type $\uua$. 
    We say that $\ua_a$ is {\sl  $c$-contained} in $\ua_b$ if $\G-D_{\ua_a}^c\subseteq \G-D_{\ua_b}^c $, i.e., $D_{\ua_b}^c\subseteq D_{\ua_a}^c $.
\end{defi}

\begin{defi}[$\us$-innermost]
    Let $\underline{\ua}$ be a vertex type, $\us_c$ a based partition and $\ua_a$ a based partition in $\uua$. We say that $\ua_a$ is $\us_c$-{\sl innermost} in $\underline{\ua}$ if it crosses $\us_c$ and is minimal with respect to the partial order of $c$-containment in the set of based partitions of $\underline{\ua}$ that cross $\us_c$.
\end{defi}

\begin{lem}\label{lem:innermost} Let $\underline{\ua}$ be a vertex type, $\us_c$ a based partition, $\ua_a$ a based partition of $\uua$ and $\ua_a'$ the refinement of $\ua_a$ respect to $\us_c$.
\begin{itemize}
    \item[i)] If $\ua_a$ is $\us_c$-innermost, then $\ua_a'$ is compatible with all the based partitions of $\uua$ and  if $\ua_b$ is another based partition of $\uua$ which is also $\us_c$-innermost, then the refinement $\ua_b'$ of $\ua_b$ respect to $\us_c$ is compatible with $\ua_a'$.

    \item[ii)] If $\ua_a$ is not $\us_c$-innermost and $\ua_x$ is a based partition of $\uua$ which also crosses $\us_c$ and is properly $c$-contained in $\ua_a$, then $\ua_a'$ and $\ua_x$ are not compatible.
\end{itemize}

\end{lem}
\begin{proof} We begin with item i). Let $\ua_b$ be a based partition of $\uua$. If $b\in\st_\G(a)$ then it is obvious that $\ua_a'$ and $\ua_b$ are compatible.
If $\ua_b$ and $\us_c$ are compatible, then $\ua_a'$ and $\ua_b$ are compatible because of Lemma \ref{lem:collComp}.
In other case, set
$$\ua_a=\{\{a\},D^c_{\ua_a},P_1,\ldots,P_t\}$$
so $c\in D^c_{\ua_a}$. We claim that also $b$ lies in $D^c_{\ua_a}$. This will imply $D^c_{\ua_a}=D^b_{\ua_a}$ and as refinement of $\ua_a$ respect to $\us_c$ does not affect $D^c_{\ua_a}$ we deduce, using Lemma \ref{lem:charcr}, that $\ua_a'$ and $\ua_b$ are compatible.

Assume for a contradiction that $b\not\in D^c_{\ua_a}$. As $b$ is not adjacent to $a$, it must be in some of the petals of $\ua_a$, say $b\in P_1$. Therefore $P_1=D^b_{\ua_a}$. As $\ua_a$ and $\ua_b$ are compatible, Lemma \ref{lem:compatible} yields
$$D^c_{\ua_a}\cup P_1\cup\cdots\cup P_t\subseteq D^a_{\ua_b}$$
so $c\in D^a_{\ua_b}$ and therefore $D^a_{\ua_b}=D^c_{\ua_b}$. Moreover, $a\in D^a_{\ua_b}$ but $a\not\in D^c_{\ua_a}$ so we have 
$$D^c_{\ua_a}\subsetneq D^a_{\ua_b}.$$
Taking complementaries, this contradicts the fact that $\ua_a$ is $\us_c$-innermost.

Finally, assume that $\ua_b$ is also $\us_c$-innermost. Put
    $$\ua_b=\{\{b\},D^c_{\ua_b},Q_1,\ldots,Q_k\}.$$
    The argument above implies $D^c_{\ua_a}=D^b_{\ua_a}$ and $D^c_{\ua_b}=D^a_{\ua_b}$ so in the process of refining $\ua_a$ and $\ua_b$ respect to $\us_c$, $D^b_{\ua_a}$ and $D^a_{\ua_b}$ are not affected and, using again Lemma \ref{lem:charcr}, we see that $\ua_a'$ and $\ua_b'$ are compatible.

    For item ii), set
   
    $$\ua_x=\{\{x\},D^c_{\ua_x},T_1,\ldots,T_r\}$$
   
As $\ua_x$ is $c$-contained in $\ua_a$, we have
$$\{x\}\cup T_1\cup\cdots\cup T_r\subseteq\{a\}\cup P_1\cup\cdots\cup P_t.$$
Therefore $x$ lies in one of the $P_i$'s and we may assume $x\in P_1$ so $P_1=D^x_{\ua_a}$.
In particular, note that $x$ and $a$ are not adjacent in $\G$. The fact that $\ua_a$ and $\ua_x$ are compatible (they are based partitions of the same vertex type) implies by Lemma \ref{lem:compatible} that $P_1=D^x_{\ua_a}$ is the union of $\st_\G(x)-\lk_\G(a)$ and all the petals of $\ua_x$ except of the one containing $a$. As $c$ can not lie in $P_1$, this implies that $D^c_{\ua_x}$ is the only petal of $\ua_x$ which  is not in $P_1$, so $a\in D^c_{\ua_x}$, i.e. $D^c_{\ua_x}=D^a_{\ua_x}$ and
\begin{equation}\label{eq:p1}P_1=(\st_\G(x)-lk_\G(a))\cup T_1\cup\cdots\cup T_r.\end{equation}
Using Lemma \ref{lem:compatible}, a necessary condition for $\ua_a'$ and $\ua_x$ to be not compatible is that  when we refine $\ua_a$ respect to $\us_c$, the petal $P_1$ gets (properly) divided, equivalently, that there is some petal $H$ of $\us_c$ such that $a\not\in H$ and 
\begin{equation}\label{eq:petalH}
\emptyset\neq H\cap P_1\subsetneq P_1.    
\end{equation}
But in fact in this case this condition is also sufficient because if there is such an $H$, then (\ref{eq:p1}) implies that one of the new components does not contain $x$ and meets one of the $T_i$'s.

Taking into account that $\ua_x$ and $\us_c$ cross, we see that there is some petal $H_1$ of $\us_c$ with $x\not\in H_1$ and non empty intersection with some of the $T_i$'s, for example we may assume $H_1\cap T_1\neq\emptyset$. As $x\in P_1$ and $T_1\subseteq P_1$ we have
$$\emptyset\neq H_1\cap T_1\subseteq H_1\cap P_1\subsetneq P_1.$$
Now, we distinguish two cases:

- If $a\in H_1$, then there must be some $H_2\neq H_1$ petal of $\us_c$ with $x\in H_2$ so $x$ lies in the intersection $H_2\cap P_1$. As $H_1\cap P_1\neq\emptyset$, we have $H_2\cap P_1\subsetneq P_1$ so we get (\ref{eq:petalH}) with $H=H_2$.

- If $a\notin H_1$, we get (\ref{eq:petalH}) with $H=H_1$.

\end{proof}

\begin{nota}\label{rem:sccontained} Observe that a consequence of the proof of item ii) in Lemma \ref{lem:innermost} is that if $\ua_x,\ua_a$ are based partitions of a vertex type $\uua$ such that both cross a based partition $\us_c$ and $\ua_x$ is strictly $c$-contained in $\ua_a$, then $x$ and $a$ can not be adjacent in $\Gamma$ and moreover $x$ and all the petals of $\ua_x$ different from $D^c_{\ua_x}$ are inside a single petal $P_1$ of $\ua$ which is also different from $D^c_{\ua_a}$. This can be represented by flower diagrams as in Figure \ref{fig:innermost} where the infinite components for both $\ua_a$ and $\ua_x$ correspond to the petals containing $c$. Note that moreover, the petal $P_1$ crosses some petal of $\us_c$ (one of the petals $H_1$ or $H_2$ at the end of the proof).
\end{nota}

    \begin{figure}
  \makebox[\textwidth][c]{ 
  \begin{tikzpicture}
  \draw (3.7,2) node[minimum size=1cm,circle,draw] (P) {$a$};
       \draw (P.80) to[out=50,in=-10,looseness=10]  node[pos=0.1](E1){} node[pos=0.4](E2){} node[pos=0.8](E3){} node[pos=0.9](E4){}  (P.-40);
\node[fit=(E1)(E2)(E3)(E4)]{};      
 \draw (P.180) to[out=185,in=295,looseness=34]  node[pos=0.0](F1){} node[pos=0.1](F2){} node[pos=0.85](F3){} node[pos=0.9](F4){}  (P.-100);
\node[fit=(F1)(F2)(F3)(F4)]{};   
     \draw (1.5,2.8) node[minimum size=1cm,circle,draw] (Q) {$c$};
      \draw (Q.-30) to[out=-23,in=-90,looseness=18]  node[pos=0.3](G1){} node[pos=0.6](G2){} node[pos=0.8](G3){} node[pos=0.9](G4){}  (Q.-90);
\node[fit=(G1)(G2)(G3)(G4)]{};
   
        \draw (1.5,-0.5) node[minimum size=1cm,circle,draw] (O) {$x$};
       \draw (O.80) to[out=80,in=15,looseness=10]  node[pos=0.1](C1){} node[pos=0.4](C2){} node[pos=0.8](C3){} node[pos=0.9](C4){}  (O.-20);
\node[fit=(C1)(C2)(C3)(C4)]{};   
     \draw (9.7,2.8) node[minimum size=1cm,circle,draw] (Q) {$c$};
      \draw (Q.-54) to[out=-59,in=265,looseness=34.2]  node[pos=0.3](G1){} node[pos=0.4](G2){} node[pos=0.6](G3){} node[pos=0.9](G4){}  (Q.-100);
\node[fit=(G1)(G2)(G3)(G4)]{};

\draw (10.1,2.5) to [out=-30,in=100, looseness=0.8]               (11.1,0.7);
\draw (11.1,0.7) to [out=-80,in=80, looseness=0.8]               (11.03,-1);
\draw (11.03,-1) to [out=-100,in=145, looseness=1]               (11.13,-1.33);
\draw (11.13,-1.33) to [out=-35,in=170, looseness=0.8]               (11.8,-1.55);
\draw (11.8,-1.55) to [out=-10,in=200, looseness=0.3]               (12.5,-1.6);
\draw (12.5,-1.6) to [out=5,in=260, looseness=1.1]               (13.1,-0.9);
\draw (13.1,-0.9) to [out=80,in=-105, looseness=0.8]               (13.3,0);
\draw (13.3,0) to [out=75,in=285, looseness=1]               (13.4,1.8);

\draw (13.4,1.8) to [out=105,in=-40, looseness=0.8]               (13.15,2.35);

\draw (13.15,2.35) to [out=140,in=-10, looseness=1]               (12.5,2.7);

\draw (12.5,2.7) to [out=170,in=10, looseness=0.8]               (11.7,2.75);

\draw (11.7,2.75) to [out=-180,in=20, looseness=0.8]               (11,2.65);

\draw (11,2.65) to [out=200,in=-10, looseness=1]               (10.19,2.68);
   
        \draw (10.5,-0.5) node[minimum size=1cm,circle,draw] (O) {$x$};
       \draw (O.50) to[out=13,in=-22,looseness=11]  node[pos=0.1](C1){} node[pos=0.2](C2){} node[pos=0.6](C3){} node[pos=0.8](C4){}  (O.-30);
\node[fit=(C1)(C2)(C3)(C4)]{};      
 
 \draw (12.7,2) node[minimum size=1cm,circle,draw] (P) {$a$};
       \draw (P.80) to[out=50,in=-10,looseness=10]  node[pos=0.1](E1){} node[pos=0.4](E2){} node[pos=0.8](E3){} node[pos=0.9](E4){}  (P.-40);
\node[fit=(E1)(E2)(E3)(E4)]{};      
 \draw (P.180) to[out=185,in=295,looseness=34]  node[pos=0.0](F1){} node[pos=0.1](F2){} node[pos=0.85](F3){} node[pos=0.9](F4){}  (P.-100);
\node[fit=(F1)(F2)(F3)(F4)]{};   
 \draw (20,2) node[minimum size=1cm,circle,draw, color=white]  {};
\end{tikzpicture}} 
\vspace{-3.75cm}
\caption{Crossings between $\us_c$, $\ua_a$ and $\ua_x$}\label{fig:innermost}
\end{figure}

Using Lemma \ref{lem:innermost} i) we deduce that the following notion is well defined.

\begin{defi}[Refinement and disjunction of vertex types] Let $\us_c$ be a based partition.
    The $\us_c$-refinement of a vertex type $\underline{\ua}$ is the result of refining the $\us_c$-innermost based partitions of $\underline{\ua}$ with respect to $\us_c$. In the same way, $\us_c$-disjunction of $\underline{\ua}$ is the result of disjoining the $\us_c$-innermost based partitions of $\underline{\ua}$ with respect to $\us_c$.
    If $\us_c$ and $\uua$ are $\G$-valid, then so is the $\us_c$-refinement of $\uua$. 
\end{defi}

\begin{lem}\label{lem:innermostorder} Let $\us_c$ be a based partition and $\uua\leq\uub$ vertex types. Let
$$\mathcal{I}(\uua)=\{a\in V_\G\mid\ua_a\text{ is $\us_c$-innermost in }\uua\},$$
$$\mathcal{I}(\uub)=\{a\in V_\G\mid\ub_a\text{ is $\us_c$-innermost in }\uub\}.$$
Then:
\begin{itemize}
    \item[i)] $\mathcal{I}(\uub)\subseteq\mathcal{I}(\uua)\cup\{a\in V_\G\mid\ua_a\text{ is compatible with }\us_c\}$ and

    \item[ii)] if $\cro(\uua,\us_c)=\cro(\uub,\us_c)$ then $\mathcal{I}(\uub)=\mathcal{I}(\uua)$.
   
\end{itemize}
\begin{proof}
    We begin with item i). Assume, for a contradiction, that $a\in\mathcal{I}(\uub)$ but $\ua_a$ is not compatible with $\us_c$ and also not $\us_c$-innermost in $\uua$. Then there is a based partition $\ua_x$ of $\uua$ properly  $c$-contained in $\ua_a$ that also crosses $\us_c$. For the $\us_c$-refinements we have $\ua_a'\leq\ub_a'$.  Moreover, the hypothesis that $\ub_a$ is $\us_c$-innermost and Lemma \ref{lem:innermost} imply that $\ub_a'$ is compatible with $\ub_x$. As $\ua_x\leq\ub_x$ and $\ua_a'\leq\ub_a'$, we deduce that $\ua_a'$ is also compatible with $\ua_x$. But this contradicts Lemma \ref{lem:innermost}.
    
For item ii), observe first that for any $a\in\mathcal{I}(\uub)$, if $a\not\in\mathcal{I}(\uua)$, then item i) implies that $\ua_a$ is compatible with $\us_c$ and as $\ub_a$ is not, we have a contradiction with the hypothesis that $\cro(\uua,\us_c)=\cro(\uub,\us_c)$. So we have $\mathcal{I}(\uub)\subseteq\mathcal{I}(\uua)$. For the converse, take $a\in\mathcal{I}(\uua)$ and assume $a\not\in\mathcal{I}(\uub)$. Then
$\ua_a$ is $\us_c$-innermost in $\uua$ but $\ub_a$ is not $\us_c$-innermost in $\uub$. Note that as $\ua_a$ is not compatible with $\us_c$, neither can be $\ub_a$. Therefore we can choose some based partition $\ub_x$ of $\uub$ which crosses $\us_c$ and is properly $c$-contained in $\ub_a$ and put
$$\ub_x=\{\{x\},D^c_{\ub_x},L_1,\ldots,L_k\}$$
where (see Remark \ref{rem:sccontained}) $\st_\G(x)-\lk_\G(a)\cup L_1\cdots\cup L_k=Q_1$ where $Q_1$ is a petal of $\ub_a$ which crosses $\us_c$. We have $\ua_x\leq\ub_x$ and the fact that $\ub_x$ crosses $\us_c$ and the hypothesis that $\cro(\uua,\us_c)=\cro(\uub,\us_c)$ imply that also $\ua_x$ crosses $\us_c$. We also have
$$\ua_x=\{\{x\},D^c_{\ua_x},J_1,\ldots,J_r\}$$
where $D^c_{\ua_x}$ is the union of $D^c_{\ua_x}$ and (possibly) some of the $L_i$'s and each $J_j$ is a union of $L_i$'s. Therefore 
\begin{equation}\label{eq:complem}\st_\G(x)-\lk_\G(a)\cup J_1\cup\cdots\cup J_r\subseteq Q_1\subseteq P_1\end{equation}
for some petal $P_1$ of $\ua_a$. If $P_1$ is the petal that contains $c$, then the fact that the petal $Q_1$ crosses $\us_c$ implies that the number of crossings should be strictly less when passing to $\ua_a$, contradicting again the hypothesis. So $P_1\neq D^c_{\ua_a}$ and therefore $P_1\subsetneq\G-D^c_{\ua_a}$ (observe that $a$ lies in the right hand set, but not in $P_1)$ which means that $\ua_x$ is (strictly) $c$-contained in $\ua_a$ and we have a contradiction.
\end{proof}
    
\end{lem}

Finally, the proof of the next result is a direct consequence of the definition.

\begin{lem}\label{lem:refMantieneOrden}
Let $\us$ be a based partition. Suppose that $\underline{\ua},\underline{\ub}$ are two vertex types with $\underline{\ua}<\underline{\ub}$ and $\mathcal{I}\subseteq V_\G$ a set such that for every $a\in\mathcal{I}$ the based partition $\ub_a$ of $\uub$ is either compatible with $\us_c$ or $\us_c$-innermost (so $\ua_a$ has the same property in $\uua$ because of Lemma \ref{lem:innermostorder}). Let $\uua'$ and $\underline{\ub}'$ be the $\us$-refinements and $\uua''$ and $\uub''$ the $\us$-disjunctions of $\uua$ and $\uub$ respect to $\mathcal{I}$. Then $\uua'\leq\uub'$ and $\uua''\leq\uub''$. 
\end{lem}

\section{Lemmas of reductivity}\label{sec:reductivity}

In this section we introduce some  notions that will be used below to prove that $\MM_\G$ is contractible. We follow \cite{MM} for definitions  but we need to provide new versions of the proofs adapted to our situation. To do that we will need to use technical results from \cite{BDay} about the action of pure symmetric automorphisms of RAAGs on graphically reduced expressions for elements  

\begin{defi}[Length]
Let $g$ be an element of $A_\G$ and $X$ a generating family. By {\sl cyclic word representative} of $g$ we mean a word $w$ in the alphabet $X^\pm$ such that some cyclic permutation of $w$ represents $g$ and by {\sl cyclic word length} $|g|_X$ of $g$ in $X$ we mean the shortest possible length of a cyclic word representative of $g$. It is easy to see that $|g|_X$ is invariant in conjugacy classes both for $g$ and $X$, i.e. $$|g|_X=|g^h|_X=|g|_{X^h}.$$

 We also say that a cyclic word representative is {\sl graphically reduced} if the word itself and all its cyclic permutations are graphically reduced. In the case when $X=V_\G$, the cyclic word length is precisely the length of any  graphically reduced cyclic word representative.

Let $u=[X,\OO]$ be a nuclear point of $\MM_\G$ and $W\subseteq A_\G$ a finite set, we define the {\sl height} of $u$ with respect to $W$ as
$$\|u\|_W=\sum_{g\in W}|g|_X.$$ 
\end{defi}

Obviously, this length is not invariant under the action of automorphisms, however we have

\begin{lem} Let $X$ be a basis of $A_\G$, then for any $g\in A_\G$ and any $\alpha\in\mathrm{Aut}(A_\Gamma)$,
$$|g|_X=|\alpha(g)|_{\alpha(X)}.$$
\end{lem}

In the rest of the section we fix a finite set $W$
 of graphically reduced cyclic word representatives.

\begin{defi}[Reductivity]\label{def:reductivity}
    Let $u=[X,\OO]$ a nuclear point of $\MM_\G$. We define the \textit{reductivity} of $\alpha\in\PAut(A_\Gamma)$ at $u$ respect to $W$ as
    $$\red_W(X,\alpha)=\|u\|_W-\|\alpha(u)\|_W=\sum_{g\in W}|g|_X-\sum_{g\in W}|g|_{\alpha(X)}$$
    and when $X=V_\G$ we will just denote 
$$\red_W(\alpha):=\red_W(V_\G,\alpha).$$
\end{defi}

In the next Lemma, we collect some easy properties for later use.

\begin{lem}\label{lem:split}\label{eq:redalpha} Let $\alpha,\beta,\gamma\in\PAut(A_\Gamma)$. Then 
\begin{itemize}
\item[i)] $\red_W(X,\alpha)=\red_{\gamma(W)}(\gamma(X),\gamma\alpha\gamma^{-1})$ and in particular, if $\gamma(X)=V_\G$, then
$\red_W(X,\alpha)=\red_{\gamma(W)}(\gamma\alpha\gamma^{-1})$, 

\item[ii)] if $\alpha$ is inner, then  $\red_W(X,\alpha)=0$ and $\red_{\alpha(W)}(X,\beta)=\red_W(\alpha(X),\beta)=\red_W(X,\beta\alpha)=\red_W(X,\beta)$,

\item[iii)] $\red_W(X,\alpha\beta)=\red_W(X,\beta)+\red_{W}(\beta(X),\alpha)$,

\item[iv)] $\red_W(\alpha\beta)=\red_W(\beta)+\red_{\beta^{-1}W}(\beta\alpha\beta^{-1})$.
   \end{itemize} 
   \end{lem}
\begin{proof} Items i) and ii) are direct consequences of the definition.
For iii),
$$\begin{aligned}
\red_W(X,\alpha\beta)=\sum_{g\in W}|g|_X-\sum_{g\in W}|g|_{\alpha\beta(X)}=
\sum_{g\in W}|g|_X-\sum_{g\in W}|g|_{\beta(X)}+\sum_{g\in W}|g|_{\beta(X)}-\sum_{g\in W}|g|_{\alpha\beta(X)}\\
=\red_W(X,\beta)+\red_{W}(\beta(X),\alpha).\\
\end{aligned}$$
Finally, iv) follows from iii) in the case when $X=V_\G$ using also that i) implies
$$\red_W(\beta(V_\G),\alpha)=\red_{\beta^{-1}W}(\beta\alpha\beta^{-1}).$$
\end{proof}

    \begin{defi}
    We say that $\alpha$ is {\sl reductive at $u=[X,\OO]$ respect to $W$} if $\red_W(X,\alpha)\geq 0$ and {\sl strictly reductive at $u$ respect to $W$} if $\red_W(X,\alpha)>0$. The \textit{reductivity of a vertex} $[X,\uua]$ of $\MM_\G$ is defined as 0 if $\uua$ is trivial and as
    $$\red_W[X,\underline{\ua}]=\max\{\red_W(X,\alpha)\mid\alpha\text{ carried by }[X,\uua],\alpha\not\in\Inn(A_\G)\}$$
in other case. If $\gamma(X)=V_\G$, then 
\begin{equation}\label{eq:redpoint}
    \red_W[X,\underline{\ua}]=\red_{\gamma(W)}[V_\G,\underline{\ua}].
    \end{equation}
 We say that a vertex $[X,\underline{\ua}]$ is {\sl reductive at $W$} if $\red_W[X,\underline{\ua}]\geq 0$ and {\sl strictly reductive}  if $\red_W[X,\underline{\ua}]> 0$. We define in the same way the reductivity of a marked based partition, i.e. as 0 in the trivial case and
    $$\red_W(X,\ua_a)=\max\{\red_W(X,\alpha)\mid\alpha\text{ carried by}(X,\ua_a),\alpha\not\in\Inn(A_\G)\}$$
    otherwise. We  say that $(X,\ua_a)$ is \textit{fully (strictly) reductive} if there exists an $\alpha\in\PAut(A_\G)$ (strictly) reductive at $X$ whose full carrier is $(X,\ua_a)$.
    We also put
    $$\red_W(\ua_a)=\red_W(V_\G,\ua_a)$$
    and again if $\gamma(X)=V_\G$, then 
    $$\red_{W}(X,\ua_a)=\red_{\gamma(W)}(V_\G,\ua_a)=\red_{\gamma(W)}(\ua_a).$$
    
\end{defi}

\begin{lem}\label{lem:ordenRed}
    If $\underline{A}_a\geq\underline{B}_a$ and $\ub_a$ is not trivial, then for any basis $X$,
    $\red_W(X,\underline{A}_a)\geq\red_W(X,\underline{B}_a)$. 
\end{lem}
\begin{proof} By Definition \ref{def:reductivity} we may assume $X=V_\G$. If $\underline{B}_a\leq \underline{A}_a$, then $\underline{A}_a$ can be obtained from $\underline{B}_a$ by dividing petals, so $\underline{A}_a$ will carry at least all the automorphisms carried by $\underline{B}_a$ (since there are less restrictions), so reductivity will at least be the same. 

\end{proof}

\subsection{The factorization and existence lemmas}
\begin{defi}
    Recall that we are denoting $L=V_\G^\pm$. We extend the notation $\lk_L(b)$ for $b\in L$ so that $\lk_L(b)$ is the set of elements $z\in L$ such that the vertex of $\Gamma$ corresponding to $z$ is adjacent to the vertex of $\G$ corresponding to $b$, $\st_L(b)$ is defined analogously. Let $w$ be a cyclically graphically reduced word in $L$. Take $c, d  \in L-\lk_L(b)$. Following \cite{BDay}, we define the adjacency counter of $w$ relative to $c,d$ and $b$, written as $\langle c, d\rangle_{w, b}$, as the number of cyclic subsegments of $w$ of the form $\left(c u d^{-1}\right)^{ \pm 1}$, where $u$ is any (possibly empty) word in $\lk_L(b)$ (here, the elements $c$, $d$, $b$ are not necessarily distinct).

For the rest of the section we fix a finite set $W$ of cyclically graphically reduced words. We define the adjacency counter of $W$ relative to $c,d$ and $b$ as:
$$
\langle c, d\rangle_{W, b}=\sum_{w\in W}\langle c, d\rangle_{w, b}
$$
For $B, C \subset L$, we define:
$$
\langle B, C\rangle_{W, b}=\sum_{c \in\left(B- \lk_L(b)\right)} \sum_{d \in\left(C- \lk_L(b)\right)}\langle c, d\rangle_{W, b}.
$$
\end{defi} 

We begin by stating a result of \cite{BDay} with our notation.

\begin{lem}[\cite{BDay} Lemma 3.16]\label{lem:Day} Let $\beta=C^b_B$ be a partial conjugation with $b\in L=V_\G^{\pm}$ and $B$ a union of connected components of $\Gamma-\st_\G(b)$. Then
\begin{equation}
    \red_W(\beta)=\langle b, B^\pm\rangle_{W, b} -\langle B^\pm, L-B^\pm-\{b\}\rangle_{W, b} \label{eq:lemIgPrev}
\end{equation}
\end{lem}
\begin{proof} Taking into account that Day denotes in \cite{BDay} the action of the partial conjugation by $b$ as $c\mapsto b^{-1}cb$, which corresponds to $\beta^{-1}$ with our notation, Lemma 3.16 of \cite{BDay} can be written as
$$\sum_{g\in W}|\beta^{-1}(g)|_{V_\G}-\sum_{g\in W}|g|_{V_\G}=
\langle B^\pm, L-B^\pm-\{b\}\rangle_{W, b}-\langle b, B^\pm\rangle_{W, b}$$
As
$$\sum_{g\in W}|\beta^{-1}(g)|_{V_\G}-\sum_{g\in W}|g|_{V_\G}=-\red_W(\beta)$$
we get the formula in the statement.
\end{proof}

 The next result is a straightforward consequence of the previous lemma and the properties of adjacency counters proved in \cite{BDay}:
\begin{lem}[\cite{BDay} Lemma 3.17]\label{lem:Day2}
    Let $\beta=C^b_B$ be a partial conjugation with $b\in L=V_\G^{\pm}$ and $B$ a union of connected components of $\Gamma-\st_\G(b)$. Then
\begin{equation}
    \red_W(\beta)=\langle b, L\rangle_{W, b} -\langle B^\pm\cup\{b\}, L-B^\pm-\{b\}\rangle_{W, b} .
\end{equation}
\end{lem}

Next, we prove a few technical Lemmas that will be needed later.

 \begin{lem}\label{lem:previoPowers}
    Let $\alpha=C^a_A$ be a partial conjugation with  $a\in L=V_\G^{\pm}$ and $A$ a union of connected components of $\Gamma-\st_\G(a)$. We have 
    \begin{itemize}
        \item[i)]if
    $
    \red_{\alpha^{-1}(W)}(\alpha)>0
    $
    then
    $
    \red_W(\alpha)>0
    $,
    \item[ii)] if
    $
    \red_{\alpha^{-1}(W)}(\alpha)\geq 0
    $
    then
    $
    \red_W(\alpha)\geq0
    $.
    \end{itemize}
    \end{lem}
\begin{proof}
    Using Lemma \ref{lem:Day2}, we can rewrite the reductivities as:
    $$
\red_{\alpha^{-1}(W)}(\alpha)= \langle a, L\rangle_{\alpha^{-1}(W), a} -\langle A^\pm\cup\{a\}, L-A^\pm-\{a\}\rangle_{\alpha^{-1}(W), a}
    $$
    and
    $$
\red_{W}(\alpha)= \langle a, L\rangle_{W, a} -\langle A^\pm\cup\{a\}, L-A^\pm-\{a\}\rangle_{W, a}
    $$
    
    We claim that 
    $$\langle a, L\rangle_{W, a}\geq \langle a, L\rangle_{\alpha^{-1}(W), a}$$ 
    and 
    $$\langle A^\pm\cup\{a\}, L-A^\pm-\{a\}\rangle_{\alpha^{-1}(W), a}\geq\langle A^\pm\cup\{a\}, L-A^\pm-\{a\}\rangle_{W, a}.$$

    For the first inequality, let $(aul^{-1})^{\pm1}$ be a subsegment counted in  $\langle a, L\rangle_{\alpha^{-1}(W), a}$.
    When we apply $\alpha$ we get a subsegment of $W$ of one of the following two types:
    \begin{itemize}
        \item $(aul^{-1})^{\pm1}$, if $l\notin A$, which is counted in $\langle a, L\rangle_{W, a}$.
        \item $(aual^{-1}a^{-1})^{\pm1}$, if $l\in A$, which contains the sub-subsegment $al^{-1}$, which is counted $\langle a, L\rangle_{W, a}$.
    \end{itemize}

    For the second inequality, let $(xuy^{-1})^{\pm1}$ be a subsegment counted in $\langle A^\pm\cup\{a\}, L-A^\pm-\{a\}\rangle_{W, a}$.
    When we apply $\alpha^{-1}$ we get a subsegment of $\alpha^{-1}(W)$ of one of the following two types:
    \begin{itemize}
        \item  $(a^{-1}xauy^{-1})^{\pm1}$ if $x\neq a$, and we can find the sub-subsegment $auy^{-1}$ which is counted in $\langle A^\pm\cup\{a\}, L-A^\pm-\{a\}\rangle_{\alpha^{-1}(W), a}$.
        \item $(auy^{-1})^{\pm1}$, if $x=a$, which is counted in $\langle A^\pm\cup\{a\}, L-A^\pm-\{a\}\rangle_{\alpha^{-1}(W), a}$.
    \end{itemize}

    In case i), we obtain the next chain of inequalities
    $$
\langle a, L\rangle_{W, a}\geq \langle a, L\rangle_{\alpha^{-1}(W), a} > \langle A^\pm\cup\{a\}, L-A^\pm-\{a\}\rangle_{\alpha^{-1}(W), a}\geq\langle A^\pm\cup\{a\}, L-A^\pm-\{a\}\rangle_{W, a},
    $$
    which proves the first item. The proof of the second item is analogous.
\end{proof}

\begin{lem}\label{lem:powers} Let $\alpha=C^a_A$ be a partial conjugation with  $a\in L=V_\G^{\pm}$ and $A$ a union of connected components of $\Gamma-\st_\G(a)$. For any $k>0$ we have
\begin{itemize}
\item[i)] if $\red_W\alpha^k>0$, then $\red_W\alpha>0$,

\item[ii)] if $\red_W\alpha^k\geq 0$, then $\red_W\alpha\geq 0$.
\end{itemize}
\end{lem}
\begin{proof} By Lemma \ref{lem:split} iii), we have
    $$
\red_W(\alpha^k)=\red_W(\alpha)+\red_W(\alpha(V_\G), \alpha^{k-1})=\red_W(\alpha)+\red_W(\alpha(V_\G), \alpha)+\cdots+\red_W(\alpha^{k-1}(V_\G), \alpha),
    $$
    which, using Lemma \ref{lem:split} i), can be rewritten as
    $$
\red_W(\alpha^k)=\red_W(\alpha)+\red_{\alpha^{-1}(W)}( \alpha)+\cdots+\red_{\alpha^{-k+1}(W)}( \alpha).
    $$
Since $\red_W(\alpha^k)>0$, we have that at least one of the addends must be positive. If it is $\red_W(\alpha)$ we are done. In other case, there is some $j>0$ with $\red_{\alpha^{-j}(W)}( \alpha)>0$. Now, let $W'=\alpha^{-j+1}(W)$, we have that $\alpha^{-j}(W)=\alpha^{-1}(W')$, so we can apply Lemma \ref{lem:previoPowers} i) and deduce $\red_{\alpha^{-j+1}(W)}( \alpha)>0$. We can repeat this procedure until we get $\red_W(\alpha)>0$.

To prove iii), we work in the same way, but using Lemma \ref{lem:previoPowers} ii).
\end{proof}

\begin{lem}\label{lem:newCZ} Let $\alpha=(C^a_D)^k,\beta=(C^b_D)^t$ where $C^a_D$ and $C^b_D$ are partial conjugations with $a,b\in L=V_\G^{\pm}$ and $D$ is a union of shared connected components of $\Gamma-\st_\G(a)$ and $\G-\st_\G(b)$ and $k,t>0$. Then 
\begin{itemize}
\item[i)] if $\red_W\alpha>0$, then $\red_W\beta<0$,

\item[ii)] if $\red_W\alpha\geq 0$, then $\red_W\beta\leq 0$.
\end{itemize}
\end{lem}
\begin{proof} Assume first $k=t=1$. Using Lemma \ref{lem:Day}, to prove i) we have to check that $$\langle a, D^\pm\rangle_{W, a} -\langle D^\pm, L-D^\pm-\{a\}\rangle_{W, a} >0$$ implies $$\langle b, D^\pm\rangle_{W, b} -\langle D^\pm, L-D^\pm-\{b\}\rangle_{W, b} <0 .$$ 
    We will prove that $\langle D^\pm, L-D^\pm-\{a\}\rangle_{W, a}\geq \langle b, D^\pm\rangle_{W, b} $, which, by symmetry, implies that $\langle D^\pm, L-D^\pm-\{b\}\rangle_{W, b}\geq \langle a, D^\pm\rangle_{W, a}.$ 

    Let $(bud^{-1})^{\pm1}$ be a subsegment counted in $\langle b, D^\pm\rangle_{W, b}$, so $u$ is a (possibly empty) word in $\lk_L(b)$ and $d\in D^{\pm 1}$. We can rewrite it as $(du^{-1}b^{-1})^{\mp1}$. Since $\st_L(b)\cap (D^{\pm 1}\cup \{a\})=\emptyset$, we have that $u^{-1}b^{-1}$ is a word in $ L-D^\pm-\{a\}$, in consequence $(du^{-1}b^{-1})^{\mp1}$ is counted in $\langle D^\pm, L-D^\pm-\{a\}\rangle_{W, a}$ and we have that the inequality holds.
    So we get:
    $$
\langle D^\pm, L-D^\pm-\{b\}\rangle_{W, b}\geq \langle a, D^\pm\rangle_{W, a} > \langle D^\pm, L-D^\pm-\{a\}\rangle_{W, a}\geq \langle b, D^\pm\rangle_{W, b} 
    $$
    where the inequality in the middle comes from the assumption. Therefore
    $$
\langle b, D^\pm\rangle_{W, b}-\langle D^\pm, L-D^\pm-\{b\}\rangle_{W, b}<0
    $$
The proof of ii) is analogous.
    
Now, for the general case, put $\alpha_0=C^a_D$, $\beta_0=C^b_D$. If $\red_W(\alpha)>0$, Lemma \ref{lem:powers} implies $\red_W(\alpha_0)>0$ so $\red_W(\beta_0)<0$ and using Lemma \ref{lem:powers} again we deduce  $\red_W(\beta)<0$ so we have i). Item ii) is shown similarly.
\end{proof}

This last lemma can be seen as a variation of Lemma 3.8 of \cite{MM} that will be crucially used in the proof of Theorem B below. This variation is needed already in the case of the original McCullough-Miller complex $\MM_n$ because there is an issue with the statement of Lemma 3.8 in \cite{MM}. To explain the problem and illustrate it with an example, we will state with our notation a particular case of that result. Assume that $\alpha,\sigma$ are powers of partial conjugations of a free group with respect to a fixed basis, say $\alpha=(C^x_P)^k,\sigma=(C^y_Q)^t$ with $x\neq y$. Then the based partitions $\{\{x\},P,P_1\}$ and $\{\{y\},Q,Q_1\}$ are the full carriers of $\alpha$ and $\sigma$ respectively. Let $\alpha_0$ be the automorphism defined on the basis by $\alpha_0(v)=\alpha(v)$ if $v\in Q_1$ and $\alpha_0(v)=v$ otherwise, and define $\sigma_0$ analogously, i.e., $\sigma_0(v)=\sigma(v)$ if $v\in P_1$, $\sigma_0(v)=v$ otherwise. Then Lemma 3.8 of \cite{MM} claims that for each cyclically reduced word $w$
$$\red_w(\alpha_0)+\red_w(\sigma_0)\geq \red_w(\alpha)+\red_w(\sigma).$$

\begin{ej}\label{ex:issue} Consider $\alpha=(C^x_b)^2$ and $\sigma=C^y_b$ automorphisms of the free group of rank 4 and basis $x,y,b,c$. These automorphism have the based partitions $\{\{x\},\{b\},\{y,c\}\}$ and $\{\{y\},\{b\},\{x,c\}\}$ as full carriers. The associated automorphisms $\alpha_0$ and $\sigma_0$ as before are precisely the identity so we have
$$\red_w(\alpha_0)=\red_w(\sigma_0)=0$$
for any cyclically reduced word $w$. Now, let $w=x^2bx^{-2}c$, which is cyclically reduced. We have $\alpha^{-1}(w)=x^2x^{-2}bx^2x^{-2}c=bc$ and $\beta^{-1}(w)=x^2y^{-1}byx^{-2}c$ so
$$\red_w(\alpha)=|w|_{V_\G}-|w|_{\alpha(V_\G)}=|w|_{V_\G}-|\alpha^{-1}(w)|_{V_\G}=6-2=4,$$
$$\red_w(\sigma)=|w|_{V_\G}-|w|_{\sigma(V_\G)}=|w|_{V_\G}-|\sigma^{-1}(w)|_{V_\G}=6-8=-2.$$
Therefore
$$\red_w(\alpha)+\red_w(\sigma)>0=\red_w(\alpha_0)+\red_w(\sigma_0)$$
contradicting \cite[Lemma 3.8]{MM}. In fact, \cite[Lemma 3.8]{MM} is true in the case when $\alpha$ and $\sigma$ have no multiple factors (this can be seen using the proof of Lemma \ref{lem:newCZ}), however for the general case only the weaker statement of Lemma \ref{lem:newCZ} holds.
\end{ej}

We will need three more technical results.

\begin{lem} \label{lem:facto}
    Let $\alpha=C^a_A,\beta=C^b_B$ be partial conjugations such that $a,b\in L=V_\G^{\pm}$ and $A$, $B$ are unions of connected components of $\Gamma-\st_\G(a)$ and $\G-\st_\G(b)$ respectively. Assume that either $a\neq b$ and $a\in\st_L(b)$ or $A\cap B=A\cap\st_\G(b)=B\cap\st_\G(a)=\emptyset$. 
Then for any $k,t>0$
    $$
\red_W(\beta^k)=\red_{\alpha^t(W)}(\beta^k).
    $$

\end{lem}

\begin{proof} Assume first that $k=t=1$. Using Lemma \ref{lem:Day}, we have
\begin{equation}
    \red_W(\beta)=\langle b, B^\pm\rangle_{W, b} -\langle B^\pm, (L-B^\pm)-\{b\}\rangle_{W, b} \label{eq:lemIgPrev}
\end{equation}
and also
\begin{equation}
    \red_{\alpha(W)}(\beta)=\langle b, B^\pm\rangle_{\alpha(W), b}-\langle B^\pm, (L-B^\pm)-\{b\}\rangle_{\alpha(W), b} \label{eq:lemIgPrev2}
\end{equation}

We claim  that 
\begin{equation}\label{eq:lemIg1}
   \langle B^\pm, (L-B^\pm)-\{b\}\rangle_{W, b}=\langle B^\pm, (L-B^\pm)-\{b\}\rangle_{\alpha(W), b}. 
\end{equation}
and that
\begin{equation}\label{eq:lemIg2}
   \langle b, B^\pm\rangle_{W, b} =\langle b, B^\pm\rangle_{\alpha(W), b},  
\end{equation}
obviously these two claims imply the result.

Let $\pi:\G\to\G-\lk_\G(b)$ be the map given by $v\mapsto v$ if $v\not\in\lk_\G(b)$ and 
$v\mapsto 1$ if $v\in\lk(b)$. We extend $\pi$ to $A_\G$ in the obvious way. Observe that any of the previous counters can be counted in the projections of $g$ and $\alpha(g)$ in the group $A_{\G-\lk_\G(b)}$. In particular this implies that in the case when $a\in\lk_\G(b)$, $\pi(w)=\pi(\alpha(w))$. In other words, applying $\alpha$ has no effect on the counters and we  have both claims (\ref{eq:lemIg1}) and (\ref{eq:lemIg2}). So from now on we assume  $A\cap B=\emptyset$. Moreover note that $A_{\G-\lk_\G(b)}$ is the free product of $A_B$, $A_{V_\G-B}$ and the cyclic group generated by $b$.

We begin with (\ref{eq:lemIg1}). The left hand side counter counts subsegments of the form $cd$ or $d^{-1}c$ with $c\in B^{\pm}$ and $d^{-1}\in (L-B^\pm)-\{b\}$. By symmetry, we can consider just the case $cd$. As $a\in (L-B^\pm)-\{b\}$, we see that $d$ could be equal to $a$, so when we apply $\alpha$ it might happen that $d$ gets canceled. But if $d\neq a$, we have that $\alpha(cd)$ can be either $cd$ or $ca^{-1}da$ and in both cases we get a subsegment counted by the right hand side counter.
Finally, if $d=a$ is canceled when we apply $\alpha$, then there must be some $s\in A^{\pm}$ which is the next letter in the expression for $g$, i.e., such that we have the subsegment $cas$. Then applying $\alpha$ we get $caa^{-1}sa=csa$ and as $s\in (L-B^\pm)-\{b\}$ we  have again a subsegment counted by the right hand side counter.
Therefore we have 

$$ \langle B^\pm, (L-B^\pm)-\{b\}\rangle_{W, b}\leq\langle B^\pm, (L-B^\pm)-\{b\}\rangle_{\alpha(W), b}$$
and applying the same argument for $\alpha^{-1}$ yields
$$\langle B^\pm, (L-B^\pm)-\{b\}\rangle_{\alpha(W), b}\leq\langle B^\pm, (L-B^\pm)-\{b\}\rangle_{\alpha^{-1}\alpha(W), b}$$
and we get (\ref{eq:lemIg1}). 

For (\ref{eq:lemIg2}), note that the left hand side counters counts subsegments of the form $bc$ or $cb^{-1}$ for $c\in B^{\pm}$. These subsegments remain invariant when we apply $\alpha$ and moreover there can be no cancellation, so we get the claim. 

Next, consider the case of an arbitrary $k>0$ and $t=1$. Then, items iv) of Lemma \ref{lem:split}  implies 
$$\red_{\alpha(W)}(\beta^k)=\red_{\alpha(W)}(\beta)+\red_{\beta^{-1}\alpha(W)}(\beta^{k-1})$$
As $\alpha$ and $\beta$ commute modulo an inner automorphism, using also induction on $k$ we get
$$\red_{\beta^{-1}\alpha(W)}(\beta^{k-1})=\red_{\alpha\beta^{-1}(W)}(\beta^{k-1})=\red_{\beta^{-1}(W)}(\beta^{k-1})$$
so the case $k=t=1$ together with a new application of item iv) of Lemma \ref{lem:split} yield
$$\red_{\alpha(W)}(\beta^k)=\red_{W}(\beta)+\red_{\beta^{-1}(W)}(\beta^{k-1})=\red_W(\beta^k).$$

Finally, we consider the general case when $k,t>0$ are arbitrary. Induction on $k+t$ yields for $W'=\alpha(W)$
$$\red_{\alpha^t(W)}(\beta^k)=\red_{\alpha^{t-1}(W')}(\beta^k)=\red_{W'}(\beta^k)=\red_{\alpha(W)}(\beta^k)=\red_W(\beta^k).$$

\end{proof}

\begin{lem}[The Factorization Lemma]
   
\label{lem:Facto}
 Let $X=\gamma(V_\G)$ be a basis of $A_\G$ and $\alpha\in\PAut(A_\G)$ with a factorization $\alpha=\alpha_m\cdots\alpha_2\alpha_1$ such that each $\alpha_i$ is carried by a marked based partition $(X,\ua^i_{a_i})$. Assume that  for each $i\neq j$, either $a_i\neq a_j$ and $\ua^i_{a_i}$, $\ua^j_{a_j}$ are compatible, or $a_i=a_j$ and $\gamma\alpha_i\gamma^{-1}$ and $\gamma\alpha_j\gamma^{-1}$ have disjoint supports. Then
$$
\red_W(X,\alpha_m\cdots\alpha_2\alpha_1)=\sum_{i=1}^{m}\red_W(X,\alpha_i).
$$
\end{lem}
\begin{proof} By further factoring the $\alpha_i$'s if necessary, we may assume that each $\gamma\alpha_j\gamma^{-1}$ is indeed a power of a partial conjugation. We argue by induction on $m$. The hypothesis implies that the $\alpha_i$'s commute modulo inner automorphisms and this property is kept after the possible splitting. Using iii) of Lemma \ref{lem:split} we have
$$
\red_W(X,\alpha_m\cdots\alpha_2\alpha_1)=\red_W(X,\alpha_1)+\red_W(\alpha_1(X),\alpha_m\cdots\alpha_2)
$$
As $\alpha_1(X)=\alpha_1\gamma(V_\G)$ and the $\alpha_i$'s commute modulo inner automorphisms, we see that each $\alpha_i$ is also carried by $(\alpha_1(X),\ua^i_{a_i})$ and then by induction we have
$$
\red_W(\alpha_1(X),\alpha_m\cdots\alpha_2)=\sum_{i=2}^{m}\red_W(\alpha_1(X),\alpha_i)$$
so we only have to check that for any $i$,
$$\red_W(\alpha_1(X),\alpha_i)=\red_W(X,\alpha_i).$$
Let $W'=\gamma(W)$, $\beta_1=\gamma\alpha_1\gamma^{-1}$ and $\beta_i=\gamma\alpha_i\gamma^{-1}$. We have 
$$\begin{aligned}
\red_W(X,\alpha_i)=\red_{W'}(\beta_i)\\
\red_W(\alpha_1(X),\alpha_i)=\red_{\beta_1^{-1}(W')}(\beta_i)
\end{aligned}$$
so we may assume that $X=V_\G$ and that $\alpha_1$, $\alpha_i$ are both powers partial conjugations carried by $\ua^1_{a_1}$, $\ua^i_{a_i}$ and we have to show
\begin{equation}\label{eq:Facto}\red_{\alpha_1^{-1}(W)}(\alpha_i)=\red_W(\alpha_i)\end{equation}

If $a_1=a_i$ then, as the supports are disjoint, Lemma \ref{lem:facto} implies (\ref{eq:Facto}). If $a_i\neq a_j$, then the based partitions $\ua^1_{a_1}$ and $\ua^i_{a_i}$ are compatible and Lemma \ref{lem:compatible1} implies that there are inner automorphisms $\nu_1$ and $\nu_i$ such that $\nu_1\alpha_1^{-1}$ and $\nu_i\alpha_i$ are powers of partial conjugations that satisfy the hypothesis of Lemma \ref{lem:facto} (we allow $\nu_1$ and/or $\nu_i$ to be the identity) so we have
$$\red_{\alpha^{-1}(W)}(\alpha_i)=\red_{\nu_1\alpha^{-1}(W)}(\nu_i\alpha_i)=\red_W(\nu_i\alpha_i)=\red_W(\alpha_i)$$
and we are done.

\end{proof}

The proof of our next Lemma also uses results of \cite{BDay}, we thank Matt Day for pointing out the argument.
\begin{lem}[The Existence Lemma]
\label{lem:existence}
   Let $W$ be a set of cyclically graphically reduced words such that for the nuclear vertex $u_0=[V_\G,\OO]$ the height $\|u_0\|_W$ is not minimal. Then there exists a 
   partial conjugation $\alpha$ which is strictly reductive at $W$, i.e., with $\red_W(\alpha)>0$.
\end{lem}
\begin{proof}
    Let $u=[Y,\OO]$ be a nuclear vertex whose height $\|u\|_W$ with respect to $W$ is minimal so $\|u_0\|_W>\|u\|_W$ and take $\beta\in\PAut(A_\G)$ such that $\beta(V_\G)=Y$. For technical reasons that will be clear below, we need to work with the following modification of $W$:
    $$W'=W\cup\{b^k\mid b\in V_\G\}$$
    for some $k>0$. Then
    $$\|u_0\|_{W'}=\sum_{h\in W'}|h|_{V_\G}=\|u_0\|_W+kn$$
and the fact that $\beta$ is symmetric implies that also
    $$\|u\|_{W'}=\sum_{h\in W'}|h|_{Y}=\sum_{h\in W'}|h|_{\beta(V_\G)}=\sum_{h\in W'}|\beta^{-1}(h)|_{V_\G}=\|u\|_W+kn$$
    and we also have 
$$\|u\|_{W'}=\sum_{h\in\beta^{-1}(W')}|h|_{V_\G}.$$
    
    Theorem B of \cite{BDay} states that the automorphisms in $\langle\Omega_l\rangle$ can be peak reduced by elements in $\Omega_l$ respect to any set of cyclic word representatives $W'$. Here, $\Omega_l$ is the set of long range Whitehead automorphisms $\varphi$ of $A_\G$: these are automorphisms for which there is an $a\in L=V_\G^{\pm}$ and $A\subseteq L$ so that $\varphi$ fixes all vertices adjacent to $a$ and such that for any $x\in A$, $\varphi(x)\in\{x^a, ax,xa,x\}$. In particular $\Omega_l\cap\PAut(A_\G)$ is the set of partial conjugations. Therefore $\beta^{-1}\in\langle\Omega_l\rangle$ so it can be peak reduced. This implies that there is a factorization 
    $$\beta^{-1}=\alpha_m\ldots\alpha_1$$
    with $\alpha_i\in\Omega_l$ such that the series 
    $$\sum_{h\in\alpha_i\ldots\alpha_1(W')}|h|_{V_\G}$$
    has a particular form: it can strictly decrease at the beginning, then might remain the same for a while and then strictly increase. The fact that for the first and last values we have 
    $$\|u_0\|_{W'}=\sum_{h\in W'}|h|_{V_\G}>\|u\|_{W'}=\sum_{h\in\alpha_m\ldots\alpha_1(W')}|h|_{V_\G}$$
    implies that there must be some strict decrease, and that must happen right at the beginning, i.e.
    $$\|u_0\|_{W'}=\sum_{h\in W'}|h|_{V_\G}>\sum_{h\in\alpha_1(W')}|h|_{V_\G}=\sum_{h\in W'}|\alpha_1(h)|_{V_\G}=\sum_{h\in W'}|h|_{\alpha_1^{-1}(V_\G)}=\|\alpha_1^{-1}(u_0)\|_{W'}.$$
    If $\alpha_1$ is not a partial conjugation, then for some $a\in L$ and $b\in V_\G$, $\alpha_1(b)$ is either $ab$ or $ba$ so 
    $$|\alpha_1(b^k)|_{V_\G}-|b^k|_{V_\G}\geq  k$$
    and this means that if we chose $k\geq\|v\|_{W'}$ we get a contradiction. Therefore $\alpha_1$ must be a partial conjugation and we have
    $$\red_W(\alpha_1^{-1})=\|u_0\|_W-\|\alpha^{-1}_1(u_0)\|_W=\|u_0\|_{W'}-\|\alpha_1(u_0)\|_{W'}>0$$
    so we can take $\alpha=\alpha_1^{-1}$.
\end{proof}

\section{Proof of Theorem B}\label{sec:contractibility}

 In this section we prove Theorem B, which asserts that $\MM_\G$ is contractible. The proof follows the strategy of \cite{MM} and proceeds as follows. The complex $\MM_\G$ is the union of the starts  of the nuclear vertices, so we can form $\MM_\G$ by adding inductively these stars and show contractibility by proving that at each step the intersection of a new star with the union of the stars previously accumulated is contractible. Thus most of the argument takes place in the star of a single nuclear vertex, which is isomorphic to the simplicial realization of the $\G$-Whitehead poset $\Wh_\G$. To analyze the intersection we will use Quillen's poset Lemmas, which are stated in Subsection \ref{subsec:posetlemmas}. The main argument is given in Subsection \ref{subsec:contractibility}.

\subsection{Poset lemmas}\label{subsec:posetlemmas} Throughout this Subsection, we fix posets $X$ and $Y$. We denote by $|X|$ and $|Y|$ their respective simplicial realizations.

\begin{defi}[Poset map]
    A function $f:X\to Y$ is a poset map if it is order-preserving. A poset map induces a simplicial map between the simplicial realizations which is denoted by $|f|:|X|\to |Y|$.
\end{defi}

\begin{defi}[Upper and lower star]
     For $x\in X$ (which can also be seen as a vertex in the realization $|X|$) the upper star of $x$ is the (partially ordered) subset $\st_+(x)=\{z\in X\ |\ z\geq x\}\subseteq X$. The lower star is defined to be $\st_-(x)=\{z\in X\ |\ z\leq x\}$. The simplicial realizations are denoted $|\st_+(x)|$ and $|\st_-(x)|$; these are contractible subcomplexes of the simplicial realization $|X|$.
\end{defi}

From Chapter 1.3 of \cite{Q} we have the following lemma.
\begin{lem}
    Let $f: X \rightarrow Y$ and $g: X \rightarrow Y$ be poset maps. If $f(x) \geq g(x)$ for all $x$ in $X$ then $|f|$ is homotopic to $|g|$ relative to the subcomplex of $|X|$ spanned by $\{x \in X \mid f(x)=g(x)\}$. In particular, for $f:X\to X$, if $f(x) \geq x$ for all $x$ (or $f(x) \leq x$ for all $x$), the image $|f(X)|$ is a weak deformation retract of $|X|$. If $f$ is also idempotent (i.e., $f^2(x)=f(x)$ for any $x\in X$), then $|f(X)|$ is a strong deformation retract of $|X|$. 
    \label{lem:posetmap}
\end{lem} 

The next consequence is also a famous result by Quillen that will be very useful below.

\begin{lem}\label{lem:conically}
    Let $f: X \rightarrow X$ be a poset map. Assume that there is some $x_0\in X$ such that 
    \begin{itemize}
\item[i)] either $x_0\leq f(x) \geq x$ for all $x$,

\item[ii)] or $x_0\geq f(x) \leq x$ for all $x$.
\end{itemize}
Then $|X|$ is contractible.
\end{lem}

This is called to be  \textit{conically contractible} in \cite{Q}.
We will also use the following criterion for $|f|$ to be a homotopy equivalence; it appears as Proposition 1.6 of \cite{Q}.

\begin{lem}\label{lem:posetMap2}
    Let $f: X \rightarrow Y$ be a poset map. If the simplicial realization of 
    $$f^{-1}\left(s t_{-}(y)\right)=\{x\in X\mid f(x)\leq y\}$$ (respectively, of $f^{-1}\left(s t_{+}(y)\right)$  is contractible for each $y \in Y$, then $|f|$ is a homotopy equivalence. 
\end{lem}

\subsection{Proof of contractibility}\label{subsec:contractibility}

To add  the stars in the right order, we will use the notion of height of a nuclear vertex defined in Section \ref{sec:reductivity}. To do that, we fix the set $W_0=\left\{a_i a_j \mid i<j\right\}$ where the $a_i$'s are the vertices of $\G$, i.e. $V_\G=\{a_1,\ldots,a_n\}$.

\begin{lem}\label{lem:minred} For any nuclear vertex $u=[X,\OO]$ of $\MM_\Gamma$ we have  $\|u\|_{W_0} \geq n(n-1)$ and the equality holds if and only if $u=[V_\G,\OO]$.
\end{lem}
\begin{proof} Take $\alpha\in\PAut(\A)$ such that $V_\G=\alpha(X)$. Then
$$\|u\|_{W_0}=\sum_{g\in W_0}|g|_{\alpha^{-1}(V_\G)}=\sum_{g\in W_0}|\alpha(g)|_{V_\G}.$$
As $\alpha\in\PAut(\A)$, we have $\alpha(a_i)=h_ia_ih_i^{-1}$ for some $h_1,\ldots,h_n\in\A$ and up to a possible conjugation we may assume $h_1=1$. We may also assume that the words $h_ia_ih_i^{-1}$ are (non cyclically) reduced. Counting possible cancellations it is obvious that $$|\alpha(a_i)\alpha(a_j)|_{V_\G}=|h_ia_ih_i^{-1}h_ja_jh_j^{-1}|_{V_\G}\geq 2$$ so the first affirmation holds. For the second, observe that if the equality holds, then
$$2=|\alpha(a_1)\alpha(a_j)|_{V_\G}=|a_1h_ja_jh_j^{-1}|_{V_\G}=|h_j^{-1}a_1h_ja_j|_{V_\G}.$$ 
Let $b\in V_\G^\pm$ be the last letter in a reduced expression of $h_j$ and put $h_j=h_j'b$ so 
$$|h_j^{-1}a_1h_ja_j|_{V_\G}=|b^{-1}(h'_j)^{-1}a_1h'_ja_j|_{V_\G}=2$$
and since there are no cancellations in $h_ja_jh_j^{-1}$, we deduce that either $b$ or $b^{-1}$ must be canceled and this can only happen if $b\in\st_\G(a_1)^\pm$. By induction on the length of $h_j$ we deduce that for $j=2,\ldots,n$, $h_j$ lies in the subgroup generated by $\st_\G(a_1)$. 
Then, up to conjugation by $h_2$ we may assume $h_2=1$ (note that $h_2a_1h_2^{-1}=a_1$). Now,
$$2=|\alpha(a_2)\alpha(a_j)|_{V_\G}=|a_2h_ja_jh_j^{-1}|_{V_\G}=|h_j^{-1}a_2h_ja_j|_{V_\G}$$
implies that for $j=3,\ldots,n$, $h_j$ lies in the subgroup generated by $\st_\G(a_1)\cap\st_\G(a_2)$. Repeating the argument, we deduce that eventually we can conjugate to get the identity automorphism, i.e., $\alpha$ is inner so $v=[V_\G,\OO]$.
\end{proof}

For each $m>0$ such that the set $\Omega_m$ of nuclei of height $m$ respect to $W_0$ is non-empty we will set a good order in $\Omega_m$. Here, recall that a good order is an order such that every not empty subset has some minimal element, this order has no relation with the height nor the order in the Whitehead poset, it is just a good order that we can find due to the axiom of choice. To avoid confusion, we will use the $\prec_m$ symbol for this order. We collect now a series of definitions that will be used in the proof. 
\begin{itemize}
    \item $K_m$: Union of the (closed) stars of those nuclei of $\MM_\G$ whose height respect to $W_0$ is $\leq m$.
    \item $K_{\prec_m u}$: For a nuclear vertex $u=[X,\OO]$ of height $m$, 
    $$
    K_{\prec_m u}=K_{m-1}\cup\left(\bigcup_{\|v\|_W=m,v\prec_m u}\st(v)\right),
    $$
    in particular, if $u_0$ is the $\prec_m$-minimal element of height $m$ then $K_{\prec_m u_0}=K_{m-1}$.
    \item $\widehat{R}$ is the poset of vertices lying in $K_{\prec_m u}\cap\st(u)$, so if $|\cdot|$ denotes simplicial realization, we have $|\widehat{R}|=K_{\prec_m u}\cap\st(u)$.
    \item $R$: Subposet of $\widehat{R}$ consisting of $[X,\uua]\in\widehat{R}$ whose non trivial based partitions are reductive at $u$.
    \item $S$: Subposet of $R$ consisting of $[X,\uua]\in R$ which contain a based partition which is strictly reductive at $u$. It holds (we will prove it later) that $|S|=|R|\cap K_{m-1}$.
\end{itemize}

Since $\MM_\G=\cup K_m$ is a simplicial complex and $K_m\subseteq K_{m+1}$ for any $m\geq 0$, it is enough to prove inductively that each $K_m$ is contractible. The first step of this process is trivial: as Lemma \ref{lem:minred} implies that $u_0=[V_\G,\OO]$ is the only nuclear vertex with minimal reductivity respect to $W_0$ by  we just have the star  $\st(u_0)$ which is is contractible. So from now on we may assume that $m=\|u\|_{W_0}$ is not minimal.
We will prove by induction in the chosen well order that each $K_{\prec_m u}\cup\st(u)$ is contractible. By the induction hypothesis  we assume that $K_{\prec_m u}$ is contractible. Since $st(u)$ is contractible too, it is enough to prove that $|\widehat{R}|=K_{\prec_m u}\cap\st(u)$ is contractible. The proof will proceed as follows:
 \begin{itemize}
        \item 1st step: We will prove that $|\widehat{R}|$ is homotopy equivalent to $|R|$.
        \item 2nd step: We will prove that $|S|$ is contractible .
        \item 3rd step: We will prove that  $|R|$ is homotopy equivalent to  $|S|$, and the proof will be over.
    \end{itemize}

\begin{lem} Consider the nuclear vertex $u=[X, \OO]$ of height $m$  and suppose that $[X,\uua]$ is a vertex of $\widehat{R}$. Let $\Theta(\uua)$ be obtained by replacing each based partition $(X,\ua_a)$ of $(X,\uua)$ such that $\red_{W_0}(X,\ua_a)<0$ by the trivial based partition with the same operative factor (so $[X,\Theta(\underline{\underline{A}})]\leq [X, \underline{\underline{A}}]$). Then $[X, \Theta(\uua)]$ is in $R$.
     \label{lem:5.4}
\end{lem}
\begin{proof} Note that it is enough to show that $[X, \Theta(\uua)]$ lies in $\widehat{R}$, because by construction this implies that it also lies in $R$.
    As $[X,\uua]$ is in $\widehat{R}$, it belongs to the intersection $\st(u)\cap \st(w)$ for some nuclear vertex $w=[Y,\OO]$ of $K_{\prec_m u}$. This last fact implies that $\|w\|_{W_0}\leq\|u\|_{W_0}$. As the starts of $u=[X,\OO]$ and $w=[Y,\OO]$ intersect, by Lemma \ref{lem:intersect} there is some $\alpha\in\PAut(A_\G)$ carried by $(X,\underline{\underline{A}})$ such that $\alpha(X)=Y$ so $\alpha(u)=w$ and therefore
    $$\red_{W_0}(X,\alpha)=\|u\|_{W_0}-\|\alpha(u)\|_{W_0}\geq 0.$$
    We can decompose $\alpha=\alpha_1\cdots\alpha_n$ where each $\alpha_i$ is carried by the based partition $(X,\ua_i)$ (which are pairwise compatible). 
    
    Assume first that there is some $i$ with $\red_{W_0}(X,\alpha_i)>0$. Then $\ua_i$ is a based partition that remains invariant when passing to $\Theta(\uua)$ so $\alpha_i$ is carried by $(X,\Theta(\uua))$ and therefore
    $$[X,\Theta(\uua)]=[\alpha_i(X),\Theta(\uua)]\geq[\alpha_i(X),\OO]=\alpha_i(u)$$
    so $[X,\Theta(\uua)]$ lies in the star of $\alpha_i(u)$ and
    $$\|u\|_{W_0}-\|\alpha_i(u)\|_{W_0}=\red_{W_0}(X,\alpha_i)>0.$$
    Therefore $[X,\Theta(\uua)]$ lies in $K_{m-1}\subseteq K_{\prec_m u}$, i.e.
    $$[X,\Theta(\uua)]\in \widehat{R}.$$
    So now we may assume that $\red_{W_0}(X,\alpha_i)\leq 0$ for all $i$. Using the Factorization Lemma \ref{lem:Facto} we have
$$
\sum_{i=1}^n \red_{W_0}(X,\alpha_i)=\red_{W_0}(X,\alpha)
$$
and as $\red_{W_0}(X,\alpha)\geq 0$ we must have $\red_{W_0}(X,\alpha_i)=0$ for all $i$. Assume now that some $(X,\ua_i)$ has negative reductivity. By definition of reductivity this implies that $\alpha_i$ is inner so $\alpha_i(u)=u$ and $\alpha_i$ can be removed from $\alpha$. Repeating this argument, we eventually remove all those $\alpha_i$'s from $\alpha$ and we get $\alpha'$ carried by $[X,\Theta(\uua)]$ such that $\alpha'(u)=\alpha(u)=w$. Therefore, $[X,\Theta(\uua)]$ is in the star of $w=\alpha(u)$
and again this implies that $[X,\Theta(\uua)]$ is in $\widehat{R}$. 
\end{proof}

\begin{prop}\label{prop:rtor}
    $|\widehat{R}|$ is homotopy equivalent to $|R|$
\end{prop}
\begin{proof} Consider the map $f$ defined by sending each $v=[X,\uua]$ of $\widehat{R}$ to $f(v)=[X,\Theta(\uua)]$ where $\Theta(\uua)$ is as in Lemma \ref{lem:5.4}. Then $f$ is an idempotent poset map $f:\widehat{R}\to\widehat{R}$ with $f(v)\leq v$ for any $v\in\widehat{R}$, and $f(\widehat{R})=R$. By Lemma \ref{lem:posetmap}, this shows that $|R|$ is a strong deformation retract of $|\widehat{R}|$. 
\end{proof}

This lemma completes the first step of the proof of Theorem B. To proceed with the second step,
we take $S$ as defined before. We first check that $|S|=|R|\cap K_{m-1}$.
Clearly $|S|\subseteq |R|\cap K_{m-1}$. If $v\in |R|\cap K_{m-1}$ is a vertex, then it must be stritctly reductive, because it is also contained in the star of a nuclear vertex of at most height $m-1$. So there is a strictly reductive automorphism carried by $v$; applying the factorization Lemma \ref{lem:Facto} we see that at least one of the automorphisms carried by a based partition of $v$ must be strictly reductive, so $v\in S$ and $|S|=|R|\cap K_{m-1}$. 

Although $R$ is highly dependent on the ordering $\prec_m$ of the vertices of height $m$, the main advantage of working with $S$ is that it is independent of this ordering. In fact, for any vertex $v=[X,\uua]\geq u=[X,\OO]$, $v$ lies in $S$ if and only if all its based partitions $(X,\ua_i)$, $i\in V_\G$, are reductive and at least one of them is strictly reductive. 
Moreover, if $\gamma(X)=V_\G$, then
$$\red_{W_0}(X,\ua_i)=\red_{\gamma(W_0)}(V_\G,\ua_i)$$
and this means that for the rest of the proof we may assume that we are in the star of the nuclear vertex $u_0=[V_\G,\OO]$ but we have to consider reductivities over an arbitrary set of cyclically graphically reduced words $W=\gamma(W_0)$. Then $\gamma(u)=[\gamma(X),\OO]=u_0$ and the hypothesis that the height of $u$ is not minimal implies that the height of $u_0$ with respect to $W$ is not minimal because of the following:
 \begin{equation}\label{eq:height}
 \|u_0\|_W=\|\gamma(u)\|_{\gamma(W_0)}=\|u\|_{W_0}>\|u_0\|_{W_0}=\|\gamma(u_0)\|_{\gamma(W_0)}=\|\gamma(u_0)\|_W.\end{equation}

For technical reasons that will be clear later, we will prove that a complex more general than $|S|$ is contractible. More precisely, we will consider the simplicial realization of a poset defined as $S$ but with the extra requirement that all its vertex types have to be compatible with a prescribed family of based partitions. We will need first a technical Lemma. 

\begin{lem}
    Suppose that $\underline{Z}^1, \underline{Z}^2, \ldots, \underline{Z}^r$ is a compatible collection of fully reductive based partitions. Then there exists a fully strictly reductive based partition $\underline{P}$ which is compatible with each $\underline{Z}^i$. \label{lem:preCon}
\end{lem}
\begin{proof}
  If there is some $i$ such that $\red_W(\underline{Z}^i)>0$, then we can take $\underline{P}=\underline{Z}^i$ and we have finished. So we may assume that $\red_W(\underline{Z}^i)=0$ for any $i$.
  
  Using Lemma \ref{lem:existence} for the vertex $u_0=[V_\G,\OO]$ (recall that by (\ref{eq:height}), the height of $u_0$ with respect to $W$ is not minimal) we deduce that there is some partial conjugation $\alpha$ with $\red_W(\alpha)>0$ and we may choose $\alpha$ such that its full carrier $\underline{P}$ crosses a minimal number of the $\underline{Z}^i$'s. 
  
  If $\underline{P}$ crosses none of the $\z^i$  then it is compatible with them and the Lemma is proved. Suppose $\underline{P}$ crosses some of the $\z^i$. For each $i$, define $\underline{\z}^i$ to be the vertex type whose only nontrivial based partition is $\z^i$. Let $\underline{\z}=\underline{\z}^1\vee\cdots\vee\underline{\z}^r$, and let $\z^{i_1}\vee\cdots\vee \z^{i_k}$ be a based partition of $\underline{\z }$ which is $\underline{P}$-innermost. Then $\underline{{P}}$ crosses at least one of the $\z^{i_j}$, say $\z^{i_1}$. 

  Let $\beta$ be an automorphism having $\underline{Z}^{i_1}$ as full carrier such that $\red_W(\beta)=\red_W(\underline{Z}^{i_1})$ (so $\red_W(\beta)=0$). Let $(\underline{Z}^{i_1})'$ be the refinement of $(\underline{Z}^{i_1})$ respect to $\underline{P}$ and $\underline{P}'$ the refinement of $\underline{P}$ respect to $\z^{i_1}$.
  By multiplying $\alpha$ and $\beta$ by inner automorphisms if necessary, we may assume that $\alpha$ acts trivially on the petal of $\underline{P}$ containing the operative factor of $\underline{Z}^{i_1}$ and the same thing for $\beta$ respect to $\underline{P}$. We can factor
  $$\alpha=\alpha_1\ldots\alpha_t,$$
  $$\beta=\beta_1\ldots\beta_k$$
  where the support of the factors are the petals of the refined based partitions $\underline{P}'$ and $\underline{Z}^{i_1}$.
  Using the factorization Lemma \ref{lem:Facto}, we have
  $$
\red_W(\alpha)=\sum_{i=1}^t\red_W(\alpha_i )>0    \\\\\\\\\ \text{ and } \\\\\\\\\ \red_W(\beta)=\sum_{j=1}^k\red_W(\beta_j )=0
  $$
    The  disjunction $\underline{P}''$ of $\underline{P}$ from $\z^{i_1}$ crosses fewer of the $\z^i$ than $\underline{P}$ did, and so does the full carrier of any automorphism associated to it (the way of getting the full carrier is by combining petals, so new crosses cannot appear) so, by the choice of $\underline{P}$, we see that $\underline{P}''$ cannot be strictly reductive, i.e. $\red_W(\underline{P}'')\leq 0$. 
    
    As $\underline{P}''$ is the full carrier of the automorphism obtained from $\alpha$ by removing the factors $\alpha_i$ whose support 
    is a cross between $\up$ and $\underline{Z}^{i_1}$, this implies that there must be some $\alpha_i$ with support $D$ where $D$ is such a crossing, i.e., $D$ is a petal of both refinements $\underline{P}'$ and $(\underline{Z}^{i_1})'$ and such that
    $$\red_W(\alpha_i)>0.$$
  The fact that $\underline{Z}^{i_1}$ is the full carrier of $\beta$ implies that there is some factor $\beta_j$ having $D$ as support. Moreover, both $\alpha_i$ and $\beta_j$ are powers of partial conjugations on $D$ so we may apply Lemma \ref{lem:newCZ} and deduce
  $$\red_W(\beta_j)<0.$$
  But this means that if we remove $\beta_j$ from $\beta$ we get a strictly reductive automorphism having full carrier $\underline{P}_1$ with $\underline{P}_1\leq(\underline{Z}^{i_1})'$ which again contradicts the choice of $\underline{P}$ because that full carrier is compatible with all of $\underline{Z}^1, \underline{Z}^2, \ldots, \underline{Z}^r$.

\end{proof}

\begin{prop}[Contractibility of $S$]\label{lem:contS}
      Let $\underline{Z}^1, \underline{Z}^2, \ldots, \underline{Z}^r$ be a compatible collection of fully reductive based partitions, and let $T$ be the subposet of $S$ consisting of all those vertex types in $S$ which are compatible with $\underline{Z}^1, \underline{Z}^2, \ldots, \underline{Z}^r$. Then $|T|$ is contractible. 
\end{prop}
\begin{proof} 
    By Lemma \ref{lem:preCon}, there exists  a fully strictly reductive based partition $\up$ compatible with $\underline{Z}^1, \underline{Z}^2, \ldots, \underline{Z}^r$. We may choose such a $\up$ so that the reductivity $\red_W(\up)$ is maximal between all the  based partitions satisfying the same conditions (observe that there are only finitely many based partitions). Let $\underline{\up}$ be the vertex type with $\up$ as the only nontrivial based partition.

    We let $T_0\subseteq T$ be the subposet of $T$ consisting of those vertex types in which are also compatible with $\up$, i.e.
    $$T_0=\{\uua\in T\mid \uua\text{ compatible with }\up\}=\{\uua\in S\mid \uua\text{ compatible with }\underline{Z}^1, \underline{Z}^2, \ldots, \underline{Z}^r,\up\}.$$
    Observe that $\underline{\up}\in T_0$ so $T_0$ is not empty. We define the next poset map:
    $$
\begin{aligned}
g: T_0 &\longrightarrow T_0 \\
 \ \uua &\longmapsto \uua\vee \underline{\up} \\
\end{aligned}
$$
$g$ is well defined by definition of $T_0$. For any $\uua$ in $T_0$, we have  $g(\uua)\geq \uua$ and
 $g(\uua)\geq \underline{\up}$ so by Lemma \ref{lem:conically}, $|T_0|$ is (conically) contractible.

Now, we want to see that $T\simeq T_0$. Let 
$$m=\max\{\cro(\underline{\underline{A}},\up)\ | \ \underline{\underline{A}}\in T\},$$ and for any $0\leq k\leq m$ let 
$T_k$ be the subposet of $T$ consisting of those $\uua$ in $T$ that cross $\up$ at most $k$ times, i.e. 
$$T_k=\{\uua\in S\mid \uua\text{ compatible with }\underline{Z}^1, \underline{Z}^2, \ldots, \underline{Z}^r\text{ and }\cro(\underline{\underline{A}},\up)\leq k\}.$$
Then $T=T_m$. We claim that $|T_m|\simeq |T_{m-1}|\simeq\cdots\simeq |T_0|$ and as $|T_0|$ is contractible this will imply the result. We will argue by induction and prove that $T_k\simeq T_{k-1}$. We divide the proof in the following steps.

\medskip

\noindent{\bf Step 1}: There is a poset map $R_k:T_k\to T_k$ inducing a homotopy equivalence $|T_k|\simeq|R_k(T_k)|$ and whose image contains $T_{k-1}$, i.e., $T_{k-1}\subseteq R_k(T_k)$.

\smallskip

Let $M_k=\{\uua\in T_k\ | \ \cro(\uua,\up)=k\}$, and 
$$\mathcal{I}(\uua)=\{ a\in V_\G \ |\ \text{the based partition } \uua_a \text{ is $\up$-innermost in }\uua\}.$$
We claim first that if $\uua, \uub\in M_k$ lie in the same connected component of the geometric realization of $M_k$, then $\mathcal{I}(\uua)=\mathcal{I}(\uub)$. Without loss of generality, we can suppose that $\uua\leq\uub$, because in other case we replace one of them by the first vertex in a path connecting them and argue by induction on distance. So $\uua$ is obtained from $\uub$ by combining petals and, by definition of $M_k$, we can not decrease the number of petals that cross $\up$, i.e. $\cro(\uua,\up)=\cro(\uub,\up)$. Then Lemma \ref{lem:innermostorder} implies $\mathcal{I}(\uua)=\mathcal{I}(\uub)$. For such a connected component let $\mathcal{I}(W)$ denote this set of vertices of $V_\G$.

For $\uua\in M_k$ lying in the connected component $W$, let $\uua'$ be the result of refining $\uua$ respect to $\up$ at each $a\in\mathcal{I}(W)$. We define the poset map $R_k$ as follows:
    $$
\begin{aligned}
 R_k: T_k &\longrightarrow T_k \\
  \quad \quad \quad \uua &\longmapsto \begin{cases}
    \uua'\text{ if } \uua\in M_k, \\
    \uua \text{ otherwise}
\end{cases} \\
&
\end{aligned}
$$
Observe that by definition of $R_k$, if $\uua\in T_{k-1}$, then $R_k(\uua)=\uua$ so obviously $T_{k-1}$ lies inside the image of $R_k$.
We check now that $R_k$ is a well defined poset map. In other words, we have to see that $T_k$ respects the order relation and that whenever $\uua\in T_k$, then also $R_k(\uua)\in T_k$. 
We begin with the last statement. When $R_k(\uua)=\uua$, there is nothing to check. When $R_k(\uua)=\uua'$, we have to see that the based partitions of $\uua'$ are reductive and at least one is strictly reductive, which is obvious since the same holds true for $\uua$ and refinement does not decrease reductivity (when dividing petals the set of carried automorphisms can only increase, see Lemma \ref{lem:ordenRed}). The fact that the based partitions of $\uua'$ are all compatible with $\underline{Z}^1, \underline{Z}^2, \ldots, \underline{Z}^r$ follows from Lemma \ref{lem:collComp}, since the same is true for $\uua$ and refining respect to $\up$ can not increase the number of crossings with $\up$ so we have $R_k(\uua)\in T_k$. Finally, to see that it is a poset map, let $\uua\leq\uub$ both in $T_k$. If $\uua\in M_k$, then as the number of crossings of $\uub$ with $\up$ can only increase with respect to $\uua$ and $k$ is the maximum possible, we deduce $\uub\in M_k$ and then Lemma \ref{lem:refMantieneOrden} implies $R_k(\uua)=\uua'\leq\uub'=R_k(\uub)$. If $\uua\not\in M_k$, then $\uua=R_k(\uua)\leq\uub\leq R_k(\uub)$ so again $R_k$ keeps the order relation and it is a poset map.

The claim that $R_k$ induces an homotopy equivalence between the simplicial realizations $|T_k|\simeq|R_k(T_k)|$ follows from Lemma \ref{lem:posetmap} because for any $\uua\in T_k$, we have $\uua\leq R_k(\uua)$.

\medskip

\noindent{\bf Step 2}: Let $\iota:T_{k-1}\hookrightarrow R_k(T_k)$ be the inclusion map and take $\uua\in R_k(T_k)-T_{k-1}$. Then there is a poset map
$$R_{\uua}:\iota^{-1}_{\leq}(\uua)\to\iota^{-1}_{\leq}(\uua)$$
where $\iota^{-1}_{\leq}(\uua)=\{\uub\in T_{k-1}\mid \uub\leq\uua\}$ which induces a homotopy equivalence between the simplicial realizations $|\iota^{-1}_{\leq}(\uua)|\simeq|R_{\uua}(\iota^{-1}_{\leq}(\uua))|$.

\smallskip
Observe first that as $\uua\in T_k-T_{k-1}$, we have $\cro(\uua,\up)=k$. We set
$$
\begin{aligned}
 R_{\uua}: \iota^{-1}_{\leq}(\uua) &\longrightarrow \iota^{-1}_{\leq}(\uua) \\
\uub &\longmapsto \uub'
\end{aligned}
$$
where $\uub'$ denotes refinement respect to $\up$ of the based partitions $\ub_a$ of $\uub$ such that $a\in\mathcal{I}(\uua)$ (as defined in step 1).

We have to check that $R_{\uua}$ is a well defined poset map. As $\uub\leq\uua$, Lemma \ref{lem:innermostorder} implies that the refinement $\uub'$ is well defined. The based partitions of $\uub'$ are reductive and at least one is strictly reductive because the same is true for $\uub$ and refinement does not decrease reductivity. Analogously, Lemma \ref{lem:collComp} implies that all the based partitions of $\uub'$ are compatible with all of $\underline{Z}^1, \underline{Z}^2, \ldots, \underline{Z}^r$. As refining respect to $\up$ can not increase the number of crossings with $\up$ we see that $\uub'\in T_{k-1}$ and moreover since the based partitions $\ua_a$ of $\uua$ with $a\in\mathcal{I}(\uua)$ are already refined, we have $\uub'\leq\uua'=\uua$. Therefore $R_{\uua}(\uub)\in\iota^{-1}_{\leq}(\uua)$. Lemma \ref{lem:refMantieneOrden}
implies that $R_{\uua}$ is a poset map.

Again, we have $\uub\leq R_{\uua}(\uub)$ for any $\uub\in\iota^{-1}_{\leq}(\uua)$ so using Lemma \ref{lem:posetmap} we see that $R_{\uua}$ induces an homotopy equivalence between the simplicial realizations $|\iota^{-1}_{\leq}(\uua)|\simeq|R_{\uua}(\iota^{-1}_{\leq}(\uua))|$.

\medskip

\noindent{\bf Step 3}: For $\uua$ and $R_{\uua}$ as in step 2, $|R_{\uua}(\iota^{-1}_{\leq}(\uua))|$ is contractible.

\smallskip

We begin by proving the following:

Claim (*): Let $\uub$ be a vertex type with $\uub\leq\uua$ such that all its based partitions $\ub_a$ with $a\in\mathcal{I}(\uua)$ are refined respect to $\up$ and let $\uub''$ be the $\mathcal{I}(\uua)$-disjunction of $\uub$ respect to $\up$. If all the based partitions of $\uub$ are reductive and at least one is strictly reductive, then the same holds true for $\uub''$. 

Proof of (*): Let $\ub_a$ be a non trivial based partition of $\uub$. If $a\not\in\mathcal{I}(\uua)$, then $\ub_a''=\ub_a$ and they have the same reductivity. Assume now that $a\in\mathcal{I}(\uua)$.
Recall that $\up$ is a fully strictly reductive based partition, so we may choose $\alpha\in\PAut(A_\G)$ such that $\red_W(\alpha)=\red_W(\up)>0$ and $\up$ is the full carrier of $\alpha$. Let $\up'$ be the refinement of $\up$ respect to $\ub_a$, then Lemma \ref{lem:collComp} implies that $\up'$ is also compatible with all of $\underline{Z}^1, \underline{Z}^2, \ldots, \underline{Z}^r$. By multiplying $\alpha$ by an inner automorphism if necessary, we may assume that $\alpha$ acts trivially on the petal of $\up'$ containing $a$. We may factor
$\alpha=\alpha_1\ldots\alpha_k$ so that the support of each $\alpha_i$ is one of the (different) petals of $\up$ not containing $a$, then these automorphisms are powers of partial conjugations on the corresponding petal. The factorization Lemma \ref{lem:Facto} implies that
$$\red_W(\alpha)=\sum_j\red_W(\alpha_j)>0.$$
Assume that $\red_W(\alpha_j)<0$ for some $j$ and let $\gamma_j$ be the product of the rest of factors of $\alpha$, so $\alpha=\alpha_j\gamma_j$. Then
$$\red_W(\alpha)=\red_W(\alpha_j)+\red_W(\gamma_j)<\red_W(\gamma_j)$$
and as the full carrier $\up^j$ of $\gamma_j$ satisfies $\up^j\leq \up'$, it is also compatible with all of $\underline{Z}^1, \underline{Z}^2, \ldots, \underline{Z}^r$. Moreover we have $\red_W(\up^j)>\red_W(\up)$ which contradicts the choice of $\up$. Therefore, we must have $\red_W(\alpha_j)\geq 0$ for any $j$. 

Now, take $\beta\in\PAut(A_\G)$ such that $\red_W(\beta)=\red_W(\ub_a)$. Again, by multiplying by an inner automorphism if necessary, we may assume that $\beta$ acts trivially on the petal of $\ub_a$ that contains $c$, the operative factor of $\up$. We may factor $\beta$ as a product of automorphisms supported in the petals of $\ub_a$ as follows:
$$\beta=\beta_1\ldots\beta_r\nu_1\ldots\nu_l$$
so that the support of each $\beta_i$ is a petal that does not cross $\up$ and the support of each $\nu_i$ is a petal that does cross $\up$. As $\ub_a$ is refined respect to $\up$, this means that the support of $\nu_i$ equals a petal of the refinement $\up'$. Then, taking into account that $\up$ is the full carrier of $\alpha$, we see that there must be some $j$ such that $\alpha_j$ and $\nu_i$ have the same support. Then, as $\red_W(\alpha_j)\geq 0$, Lemma \ref{lem:newCZ} implies $\red_W(\nu_i)\leq 0$. 
Therefore, 
$$\red_W(\beta)=\red_W(\beta_1\ldots\beta_r)+\sum_i\red_W(\nu_i)\leq\red_W(\beta_1\ldots\beta_r)$$
and as the product $\beta_1\ldots\beta_r$ is carried by $\ub_a''$, we deduce 
$$\red_W(\ub_a)\leq\red_W(\ub''_a)$$
so we get the claim.

In particular, (*) implies $\uua''\in T_{k-1}$ because there is at least one element in $\mathcal{I}(\uua)$, so $\cro(\uua'',\up)<\cro(\uua,\up)$. In fact, $\uua''$ lies in $ R_{\uua}(\iota^{-1}_{\leq}(\uua))$.
Consider the map
 $$
\begin{aligned}
D_{\uua}:R_{\uua}(\iota^{-1}_{\leq}(\uua)) &\longrightarrow R_{\uua}(\iota^{-1}_{\leq}(\uua)) \\
 \uub& \longmapsto \uub''\\
\end{aligned}
$$
where we keep the notation $\uub''$ for the $\mathcal{I}(\uua)$-disjunction of $\uub$ respect to $\up$, i.e., the disjunction respect to $\up$ of the based partitions $\ub_a$ of $\uub$ such that $a\in\mathcal{I}(\uua)$.
We have to check first that $D_{\uua}$ is well defined. Let $\uub\in R_{\uua}(\iota^{-1}_{\leq}(\uua))$, then the vertex type $\uub''$ is well defined and the based partitions of $\uub''$ are all compatible with all of $\underline{Z}^1, \underline{Z}^2, \ldots, \underline{Z}^r$ (disjunction causes no compatibility problems). Also, the number of crossings with $\up$ does not increase when passing to $\uub''$ so we have
$$\cro(\uub'',\up)\leq\cro(\uub,\up)\leq k-1.$$
It is also obvious that $\uub''\leq\uua$ because $\uub''\leq\uub\leq\uua$. Moreover, claim (*) above implies $\uub''\in T_{k-1}$ so we have $\uub''\in R_{\uua}(\iota^{-1}_{\leq}(\uua))$ and $D_{\uua}$ is well defined. It is indeed a well defined poset map by Lemma \ref{lem:refMantieneOrden}.

Now, it is important to note that all the vertices where we apply disjunction were already refined, so $\uub\geq D_{\uua}(\uub)=\uub''$.
Finally, $\uua''\geq \uub''=D_{\uua}(\uub)$ because of Lemma \ref{lem:refMantieneOrden} so $|R_{\uua}(\iota^{-1}_{\leq}(\uua))|$ is (conically) contractible by Lemma \ref{lem:conically}.

\medskip

\noindent{\bf Step 4}: End of the proof.

\smallskip

Now, we have an inclusion map $\iota:T_{k-1}\hookrightarrow R_k(T_k)$ and for any $\uua\in R_k(T_k)-T_{k-1}$, steps 2 and 3 imply

$$
|\iota^{-1}_{\leq}(\uua)|\simeq |R_{\uua}(\iota^{-1}_{\leq}(\uua))|\simeq \ast.
$$
If $\uua\in T_{k-1}$, then $\iota^{-1}_{\leq}(\uua)$ has a maximun so the simplicial realization is contractible too. Therefore using Lemma \ref{lem:posetMap2} we deduce that $\iota$ induces a homotopy equivalence $|T_{k-1}|\simeq|R_k(T_k)|$. As $|R_k(T_k)|\simeq|T_k|$ by step 1, we get $|T_k|\simeq |T_{k-1}|$, as we wanted to prove.
\end{proof}

Observe that if we take $\underline{Z}^1,\ldots,\underline{Z}^r$ as the empty family in Lemma \ref{lem:contS}, then we deduce that $|S|$ is contractible, so we have the second main step in the proof of contractibility of $\MM_\G$.
We are now ready to finish the proof, to do that we have to show that $|R|$ is homotopy equivalent to $|S|$. In this last part of the proof the argument of \cite{MM} applies to our situation without any modification, however we include it here with our notation for the reader's convenience.

Let $v$ be a vertex type in $R$, and set $R(v)=\{w\in R\mid w\geq v\}$. As $R(v)$ is bounded below, $|R(v)|$ is contractible. Let
$\{v_1,\ldots,v_k\}$ be a compatible collection of vertex types of $R$. By Lemma \ref{lem:ComSup}, the set has a unique least upper bound $\nu=v_1\vee\cdots\vee v_k$. As the $v_i$'s lie in $R$ and $w$ is the least upper bound, Lemma \ref{lem:ordenRed} implies that the non trivial based partitions of $\nu$ can not have negative reductivity, so $\nu$ lies in $R$ and therefore
$$
R(v_1)\cap\cdots\cap R(v_k)=R(\nu)
$$
which is a contractible subcomplex of $R$.

\begin{prop}
    $R$ is homotopy equivalent to $S$. \label{rtos}
\end{prop}
\begin{proof} Put
$$C_R=\{R(v)\mid v\in R\text{ and all the nontrivial based partitions of $v$ are fully reductive}\}.$$
For any $\mu\in R$, if we replace every non trivial based partition of $\mu$ by the full carrier of an automorphisms realizing the reductivity of $\mu$ we get a vertex type $v$ such that $\mu\geq v$. Moreover, $v$ lies in $R$ because by construction it has the same reductivity as $\mu$. This implies that $C_R$ is a covering of $R$ and we may consider its nerve $\mathcal{N}(C_R)$. Recall that the nerve is the poset consisting of sets of pairwise intersecting subcomplexes $R(v)\in C_R$, so we may identify the elements of $\mathcal{N}(C_R)$ with compatible collections of vertex types
    $$
\{v_1,\ldots,v_k\}
    $$
 such that the nontrivial based partitions of each $v_i$ are fully reductive. 

    Let $f_R:R\to \mathcal{N}(C_R)$ be the poset map defined as follows for each $v$ vertex of $R$:
    $$
f_R(v)=\{\nu\in R\ |\ \nu\leq v \text{ and the nontrivial based partitions of } \nu \text{ are fully reductive}\}
    $$
     We have
    $$
|f_R^{-1}(st_+(\{v_1,\ldots,v_k\}))|=|\{v\in\R\ |\ v\geq v_i\ \text{ for any }i\}|=R(v_1\vee \cdots\vee v_k)
    $$
Therefore  Lemma \ref{lem:posetMap2} implies that $|f_R|$ is a homotopy equivalence. Consider now the restriction $f_S$ of $f_R$ to $S$.  
    Now we have
    $$
|f_S^{-1}(st_+(\{v_1,\ldots,v_k\}))|=|\{v\in S\ |\ v\geq v_i\  \text{ for any }i\}|=|st_+(v_1\vee \cdots\vee v_k)\cap S|
    $$

    We claim that $|st_+(v_1\vee \cdots\vee v_k)\cap S|$ is contractible,  using Lemma \ref{lem:posetMap2} this will imply that $|f_S|$ is a homotopy equivalence too so we will deduce $|R|\simeq|S|$.
    
    Given the compatible collection $\{v_1,\ldots,v_k\}$ of vertex types having fully reductive based partitions, let $T$ be the subposet of $S$ consisting of all vertex types which are compatible with $\{v_1,\ldots,v_k\}$. By Lemma \ref{lem:contS}, $|T|$ is contractible. There is a poset map $h:T\to T$ given by $h(v)=v\vee v_1\vee\cdots v_k$. Since $h(v)\geq v$, Lemma \ref{lem:posetmap} implies that $|T|\simeq|h(T)|$ and since $h(T)=st_+(v_1\vee\cdots\vee v_k)\cap S$ we are done.

\end{proof}

This completes step 3 of the proof. In other words, Propositions \ref{prop:rtor} and \ref{rtos}, together with Lemma \ref{lem:contS} yield Theorem B.

\section{Applications}\label{sec:applications}

\subsection{Cohomological dimension}

In this subsection we prove Theorem D. We will use the following result.

\begin{prop}\label{prop:Brown}
    Let $G$ be a group acting over an acyclic CW complex $X$, then
    $$
    \cd G\leq \mathrm{sup }\{ \cd\ G_{\sigma} + \mathrm{dim } \sigma \ |\ \sigma \textup{ a cell in } X\}
    $$
    where $G_\sigma$ is the stabilizer in $G$ of the cell $\sigma$. \label{prop:cdBound}
\end{prop}

This is a generalization of Proposition 2.4(c) in Chapter VIII of \cite{Br} (this generalization appears as an exercise), and a proof can be found in \cite{Ser}. As $\MM_\G$ is contractible, it is acyclic, and to compute the cohomological dimension of the cell stabilizers of the action of $\POut(A_\G)$ we will use the notion of rank.

\begin{defi}
    Let $\uua$ be a vertex type. The {\sl rank} of $\uua$ is
    $$r(\uua)=\sum_{a\in V_\G}\left(l(\ua_{a})-1\right).$$
     Note that this rank definition is equivalent to the rank of $\uua$ as an element of the poset $\Wh_\G$. (Recall that $l(\ua_a)$ is the number of petals of the based partition $\ua_a$).
\end{defi}

\begin{lem}\label{lem:rankstab}
    The stabilizer under the action of $\POut(A_\G)$  of a vertex $v=[X,\uua]\in\MM_\G$ is free abelian of rank $r(\uua)$.
\end{lem}
\begin{proof} We may assume that $v$ lies in the star of the nuclear vertex $[V_\G,\OO]$ because this star is a fundamental domain for the action. In other words, we may assume $X=V_\G$. The fact that the stabilizer is free abelian follows from Lemma \ref{lem:freeabelian}. The proof of that result implies that the rank of this free abelian group is precisely the sum of ranks of the free abelian groups generated by the images in $\POut(A_\G)$ of the partial conjugations carried by each of the based partitions of $\uua$. Let $\ua_a$ be one of those based partitions. The subgroup of $\PAut(A_\G)$ generated by partial conjugations carried by $\ua_a$ has rank $l(\ua_{a})$. The product of the partial conjugations based on the petals of $\ua_a$ is an inner automorphism so when we project to the outer automorphisms group we get rank $l(\ua_{a})-1$. Summing up yields $r(\uua)$.

\end{proof}

\noindent{\sl Proof of Theorem C.} Let $\sigma$ be a cell of $\MM_\G$. Then $\sigma$ is of the form
$$\sigma:v_0<v_1<\cdots<v_m$$
where $m=\mathrm{dim}(\sigma)$ and the $v_i=[X,\uua^i]$'s are vertices of $\MM_\G$ with $\uua^0<\uua^1<\cdots<\uua^m$.
Again, to compute the cohomological dimension of the stabilizer of $\sigma$ we may assume that $\sigma$ lies in the star of the nuclear vertex $[V_\G,\OO]$, so $X=V_\G$. The stabilizer $\POut(A_\G)_\sigma$ of $\sigma$ is the intersection of the stabilizers of the vertices $v_i=[V_\G,\uua^i]$, which are the subgroups of outer automorphisms carried by $\uua^i$ and by \ref{lem:ordercarrier} this is precisely the stabilizer of $v_0=[V_\G,\uua^0]$. By Lemma \ref{lem:rankstab} this group is free abelian of rank $r(\uua^0)$.
So we have
$$\cd\POut(A_\G)_\sigma+\dim\sigma=r(\uua^0)+m\leq r(\uua^m)$$
where the inequality follows from the fact that whenever $\uua^i<\uua^{i+1}$ then the rank increases at least in one, as there is at least one petal of a based partition of $\uua^i$ that splits in two in $\uua^{i+1}$.
Therefore, Proposition \ref{prop:Brown} implies
    $$
   \cd\POut(A_\G)\leq \max\{r(\uua)\mid\uua\in \Wh_\G\}.
   $$
Moreover, we can choose some vertex type $\uua$ such that $r(\uua)$ is the maximum in the right hand side. Then the stabilizer $H$ of $[V_\G,\OO]$ is a free abelian group whose rank realizes the cohomological dimension of the group $\POut(A_\G).$ 
\qed

\subsection{$\ell^2$-cohomology}
\medskip

In this subsection we will prove Theorem D which follows from the following Theorem. In \cite{McCM}, McCammond and Meier sketch a proof based on  a result outlined by
Davis and Leary in \cite{DavisLeary} and in a more general form by Davis, Januszkiewicz and Leary in
Section 3 of \cite{DavisJanusLeary}.

\begin{teo}\cite[Theorem 8.1]{McCM}
    Let $G$ be a group of finite type admitting a cocompact action on a contractible complex $X$, with strong fundamental domain $F$. Let $X[\infty]$ be the subcomplex whose isotropy groups are infinite, and let $F[\infty]=F \cap X[\infty]$. Assume
    \begin{itemize}
        \item [i)] For each 0-cell $x\in X$ the isotropy group $G_x$ is trivial or satisfies $b_p^{(2)}\left(G_x\right)=0$ for all $p \geq 0$; and
        \item[ii)] the fundamental domain $F$ is the cone over $F[\infty]$.
    \end{itemize}

Then the von Neumann dimensions of $H^i(X ; \mathcal{N}(G))$ and $\mathcal{N}(G) \otimes \overline{H}^{i-1}(F[\infty])$ are the same. If in addition to i) and ii) we have
\begin{itemize}
    \item [iii)] $G$ has no non-trivial element whose centralizer has finite index in $G$,
\end{itemize}
then the $i$-th $\ell^2$-cohomology group of $G$ is isomorphic to
$$
\ell^2(G) \otimes \overline{H}^{i-1}(F[\infty]).
$$
\end{teo}

We can now proceed with the proof of Theorem D:

\medskip

\noindent{\sl Proof of Theorem D.} Let $G=\POut(\A)$, $X=\MM_\G$ and $F=|\Wh_\G|$. We have $F[\infty]=|\Wh_\G^0|$, where $\Wh_\G^0$ is the poset obtained by removing from $\Wh_\G$ the nuclear vertex (the only one with finite stabilizer). Condition i) holds since the stabilizers are free abelian and condition ii) is obvious, so we have the first part of the statement. If, in addition, we impose condition iii) (which is not true in general for $\POut(\A)$) we obtain the last isomorphism.

\end{document}